\documentclass[letterpaper,11pt,reqno]{amsart}   	% use "amsart" instead of "article" for AMSLaTeX format
\usepackage[margin=1in]{geometry} 
\usepackage{geometry}                		% See geometry.pdf to learn the layout options. There are lots.
\usepackage{graphicx}				% Use pdf, png, jpg, or eps§ with pdflatex; use eps in DVI mode
\usepackage{marginnote}	
\usepackage{hyperref}
\usepackage{amssymb}
\usepackage{natbib}
\usepackage{amsmath, amsthm, amssymb}
 \usepackage{ifpdf}
\ifpdf
\usepackage{enumerate}
\usepackage{tikz-cd}
\usepackage{mathtools}
\usepackage{xcolor}
\usepackage{enumerate}
\usepackage[utf8]{inputenc}
\newtheorem{theorem}{Theorem}[section]
\newtheorem*{theorem*}{Theorem}

\newtheorem*{question*}{Question}

\newtheorem*{conjecture*}{Conjecture}
\newtheorem{lemma}[theorem]{Lemma}
\newtheorem{proposition}[theorem]{Proposition}
\newtheorem{corollary}[theorem]{Corollary}
\theoremstyle{remark}
\newtheorem{remark}{Remark}
\numberwithin{equation}{section}

\author{Suxuan Chen}
 \address{Department of Mathematics,  The Ohio State University}
 \email{chen.9328@osu.edu}

\title[The Hausdorff dimension of $\psi$-well approximable numbers in self-similar sets]{The Hausdorff dimension of the intersection of $\psi$-well approximable numbers and self-similar sets}
\begin{document}

\begin{abstract}

Let $\psi:\mathbb{N}\rightarrow\mathbb{R}_+$ be a monotonically non-increasing function, and let $\psi_v:\mathbb{N}\rightarrow\mathbb{R}_+$ be defined by $\psi_v(q)=1/q^v$. In this article, we consider self-similar sets whose iterated function systems satisfy the open set condition. 
For functions $\psi$ that do not decrease too rapidly, we give a conjecturally sharp upper bound on the Hausdorff dimension of the intersection of $\psi$-well approximable numbers and such self-similar sets.
When $\psi=\psi_v$ for some $v$ greater than 1 and sufficiently close to $1$, we give a lower bound for this Hausdorff dimension, which asymptotically matches the upper bound as $v\downarrow 1$.
In particular, we show that the set of very well approximable numbers has full Hausdorff dimension within self-similar sets, thus confirming a conjecture of Levesley, Salp, and Velani.

\end{abstract}
\maketitle

\section{Introduction}
Let $\psi:\mathbb{N}\rightarrow\mathbb{R}_+$ be a function. The set of $\psi$-well approximable numbers is defined by
$$W(\psi)=\big\{x\in\mathbb{R}:|qx-p|<\psi(q)\text{ for infinitely many pairs }(p,q)\in\mathbb{Z}\times\mathbb{N}\big\}.$$
For $v>0$, we define a special class of functions $\psi_v:\mathbb{N}\rightarrow\mathbb{R}_+$ by $\psi_v(q)=1/q^v$, and write $W(v)=W(\psi_v)$. A number $x\in\mathbb{R}$ is said to be very well approximable if $x$ is $\psi_v$-well approximable for some $v>1$. Denote by VWA the set of very well approximable numbers. Then, $\textup{VWA}=\bigcup_{v>1}W(v)$. We recall the following two classical results in metric Diophantine approximation regarding the size of $W(\psi)$.
\begin{theorem*}[Khintchine, \cite{Khintchine}]
Let $\psi:\mathbb{N}\rightarrow\mathbb{R}_+$ be a monotonic function. Then,
\begin{align}\label{K}
Leb(W(\psi)\cap [0,1])=
\begin{cases}
0&\text{    if     }\sum_{q=1}^\infty \psi(q)<\infty\\
1&\text{    if     }\sum_{q=1}^\infty\psi(q)=\infty
\end{cases},
\end{align}
where $Leb$ is the Lebesgue measure on $\mathbb{R}$.
\end{theorem*}

\begin{theorem*}[Jarn\'ik, \cite{Jarnik}; Besicovitch, \cite{Besicovitch}]
Let $\psi:\mathbb{N}\rightarrow\mathbb{R}_+$ be a monotonic function and $s\in(0,1)$. Then,
$$ H^s(W(\psi)\cap[0,1])=
\begin{cases}
0&\text{    if     }\sum_{q=1}^\infty q^{1-s}\psi^{s}(q)<\infty\\
\infty&\text{    if     }\sum_{q=1}^\infty q^{1-s}\psi^s(q)=\infty
\end{cases},
$$
where $H^s$ is the $s$-dimensional Hausdorff measure.
\end{theorem*}

As a refinement of Khintchine's Theorem, Jarn\'ik-Besicovitch Theorem gives more information about the size of $W(\psi)$ when $\sum_{q=1}^\infty\psi(q)<\infty$. In particular, when $v>1$, one can compute the Hausdorff dimension of $W(v)$ and get 
$$\dim_H(W(v))=\frac{2}{1+v}.$$
The notion of $\psi$-well approximable numbers extends to higher dimensions, where suitable analogues of the results of Khintchine and Jarn\'ik-Besicovitch also hold.

Recently, the question of extending the above results to intersections of $W(\psi)$ with self-similar sets has attracted a lot of attention. This consideration originated from a question by Mahler.
\begin{question*}[Mahler, \cite{Mahler}] How close can irrational elements of Cantor's set be approximated by rational numbers (1) in Cantor's set; and (2) not in Cantor's set?
\end{question*}

Here, Cantor's set is the middle-third Cantor set and is a special case of self-similar sets. It is natural to consider the generalized Mahler's question with Cantor's set replaced by self-similar sets. In this article, we are interested in irrational elements of a self-similar set in $\mathbb{R}$ that can be well approximated by rational numbers in the sense of being $\psi$-well approximable, without distinguishing whether the rational numbers are from the self-similar set or not. We interpret Mahler's question as finding the size of the intersection of $W(\psi)$ and a self-similar set.

There has been a lot of progress in this direction. We list a few such results as examples. By the work of Kleinbock, Lindenstrauss, and Weiss, cf. \cite{KLW}, if $\mu$ is a friendly measure, which includes self-similar measures, then $\mu(\textup{VWA})=0$. This generalizes an earlier work of Weiss \cite{Weiss2001}. By proving the effective equidistribution of self-similar measures along horocyclic trajectories, Khalil and Luethi show in \cite{KL} that Khintchine-type result holds for a self-similar measure $\mu$, i.e., \eqref{K} still holds with Lebesgue measure replaced by $\mu$, under the assumption that the iterated function system which determines $\mu$ satisfies the open set condition and is rational, and the Hausdorff dimension of the support of $\mu$ is large. In \cite{Datta2024}, Datta and Jana show that Khintchine-type result holds for Borel probability measures on $\mathbb{R}$ having certain Fourier asymptotics, which includes a subset of self-similar measures. By the very recent breakthrough of B\'enard, He, and Zhang \cite{BHZ}\cite{BHZ2025}, Khintchine-type result holds for all self-similar measures. We note that the results in \cite{KLW}, \cite{KL}, and \cite{BHZ2025} also apply to a more general setting where the ambient Euclidean space is $\mathbb{R}^n$.

In contrast to the Khintchine-type result, one could expect a Jarn\'ik-Besicovitch-type result for self-similar sets. In \cite{BV}, Beresnevich and Velani prove the Mass Transference Principle (MTP), which connects Lebesgue measure theoretical statements for limsup subsets of Euclidean spaces with Hausdorff measure theoretical statements for those sets, and show that Khintchine's Theorem indeed implies Jarn\'ik-Besicovitch Theorem. However, even though Khintchine's theorem is known for self-similar measures, MTP doesn't apply in this setting due to the rational approximations possibly not belonging to the support of the measure.

In \cite{LSV}, Levesley, Salp, and Velani prove that $\dim_H( \mathrm{VWA}\cap K)$ is bounded below by $\dim_H(K)/2$ for any nontrivial missing digit set $K$, which is one type of self-similar sets. When $K$ is the middle-third Cantor set, Bugeaud and Durand show in \cite{Bugeaud2016} that, for Lebesgue almost every $\alpha\in[0,1]$, $\dim(W(v)\cap (\alpha+K))$ is bounded above by $\max\{{\dim_H(K)}-1+2/{(v+1)},0\}$. 

The work in this article  is motivated by the following conjectures from the aforementioned references: 
\begin{conjecture*}[Levesley, Salp, and Velani, \cite{LSV}]
Let $K$ be the middle-third Cantor set. Then, 
\begin{align}\label{dim VWA}
\dim_H(\textup{VWA}\cap K)=\dim_H(K).
\end{align}
\end{conjecture*}
\begin{conjecture*}[Bugeaud and Durand, \cite{Bugeaud2016}]
Let $K$ be the middle-third Cantor set. Then,
\begin{align}\label{dim W(v)}
\dim_H(K\cap W(v))=\max\Big\{\dim_H(K)-1+\frac{2}{v+1},\frac{\dim_H(K)}{v+1}\Big\}.
\end{align}
\end{conjecture*}

In \cite{Yu}, Yu partially prove the above conjectures for missing digit set $K$ when the Cantor-Lebesgue measure supported on the $K$ satisfies certain regularity conditions, with additionally assuming $v$ is sufficiently close to 1 for the second conjecture. This work is further refined by Chow, Varj\'u, and Yu in \cite{CVY}, where the missing digit set $K$ for which \eqref{dim W(v)} and thus \eqref{dim VWA} hold is improved to base 7 with one digit missing.

Below, we state our main results, and we refer the readers to \S \ref{S 2} for the definitions of iterated function systems, the open set condition, and self-similar sets.

Our first result gives an upper bound for the Hausdorff dimension of the intersection of $\psi$-well approximable numbers and a self-similar set when $\psi$ is not decreasing too fast.  
\begin{theorem}\label{thm}
Let $\mathcal{F}$ be an iterated function system satisfying the open set condition, and let $K$ be the self-similar set invariant under $\mathcal{F}$. There exists a constant $J>0$, depending on the IFS $\mathcal{F}$, such that the following statement holds. Suppose $\psi$ is a monotonically non-increasing function satisfying $\sum_{q=1}^\infty\psi(q)<\infty$ and $\psi(2^n)\ge 2^{-(J+1)n}$ for all $n\in\mathbb{N}$, then, 
\begin{align}\label{thm 1.1}
\dim_H(W(\psi)\cap K)\le \inf\left\{l:\sum_{q=1}^\infty q^{s-l}\psi^{1+l-s}(q)<\infty\right\},
\end{align}
where $s$ is the Hausdorff dimension of $K$. In particular, for $1<v<J+1$, we have
$$\dim_H(W(v)\cap K)\le \dim_H(K)-1+\frac{2}{v+1}.$$
\end{theorem}

In our second result, we restrict $\psi$ to $\psi_v$ and give a lower bound for the Hausdorff dimension of the intersection of $\psi_v$-well approximable numbers and a self-similar set when $v$ is greater than 1 and close to 1.

\begin{theorem}\label{main thm}
Let $\mathcal{F}$ be an iterated function system satisfying the open set condition, and let $K$ be the self-similar set invariant under $\mathcal{F}$. There exist constants $\tilde v>0$ and $C>0$, which both depend on the IFS $\mathcal{F}$, such that
\begin{align*}
\dim_H(W(v)\cap K)\ge\dim_H(K)-C\cdot\left(1-\frac{2}{v+1}\right)
\end{align*}
holds for any $1<v<\tilde v$.
\end{theorem}

We have the following corollary, which does not require any separation conditions on the self-similar set.
\begin{corollary}\label{1.4}
Let $K$ be a self-similar set. Then, 
\begin{align*}
\dim_H(\textup{VWA}\cap K)=\dim_H(K).
\end{align*}
\end{corollary}
\begin{proof}
First, we prove the assertion when $K$ is a self-similar set whose iterated function system satisfies the open set condition.
Since $\textup{VWA}=\bigcup_{v>1}W(v)$, Theorem \ref{main thm} implies that $$\dim_{H}(\textup{VWA}\cap K)\ge\lim_{v\rightarrow 1^+}\left(\dim_H(K)-C\cdot\left(1-\frac{2}{v+1}\right)\right)=\dim_H(K).$$ And since the set $\text{VWA}\cap K$ is a subset of $K$, $\dim_H(\text{VWA}\cap K)$ is less than $\dim_H(K)$. Thus, we obtain $\dim_H(\textup{VWA}\cap K)=\dim_H(K)$.

Now, we let $K$ be a general self-similar set whose iterated function system may not satisfy the open set condition. By the first paragraph of the proof of Lemma 4.2 in \cite{Shmerkin}, there exists a sequence of self-similar sets $\{K_n\}$ whose iterated function systems satisfy the open set condition such that each $K_n$ is contained in $K$ and $\dim_H(K)-\dim_H(K_n)<1/n$. Then, we have
$$\dim_H(K)\ge \dim_H(\text{VWA}\cap K)\ge\dim_H(\text{VWA}\cap K_n)=\dim_H(K_n)\ge\dim_H(K)-\frac{1}{n}.$$
Taking the limit $n\rightarrow\infty$, we have $\dim_H(K)\ge \dim_H(\text{VWA}\cap K)\ge\dim_H(K)$. Hence, $\dim_H(K)= \dim_H(\text{VWA}\cap K)$.
\end{proof}

We note that the middle-third Cantor set doesn't satisfy the condition for \ref{dim VWA} to hold in \cite{Yu} and \cite{CVY}. Since the conclusion of Corollary \ref{1.4} requires no dimension restriction on the self-similar set, it answers the conjecture of Levesley, Salp, and Velani regarding the middle-third Cantor set affirmatively. In finding a lower bound for the Hausdorff dimension of $W(v)\cap K$, we used a refinement of the method of Yu \cite{Yu}, which in turn builds on the ideas in the proof of the Mass Transference Principle of Berenevich and Velani \cite{BV}. Another key ingredient in this paper is the effective equidistribution result of self-similar measures in \cite{BHZ} and \cite{KL}.

\subsection*{Overview}
In \S \ref{S 2.1} and \S \ref{S 2.2}, we introduce iterated function systems and the associated self-similar measures and self-similar sets. In \S \ref{equi}, we recall the result of the effective equidistribution of self-similar measures in \cite{BHZ} and \cite{KL}, which will be used in \S \ref{S 3} and \S \ref{S 4}. In \S \ref{S 3}, we prove Theorem \ref{thm}, and in \S \ref{S 4}, we prove Theorem \ref{main thm}.
\subsection*{Acknowledgements}
We are grateful to Osama Khalil for suggesting this problem and sharing many valuable insights. We also thank Michael Bersudsky, Nimish Shah, and Han Zhang for many helpful discussions.
\section{Preliminaries}\label{S 2}
\subsection{1-dimensional IFS, self-similar sets, and self-similar measures.}\label{S 2.1} To begin with, we give a brief introduction to Hausdorff measures.
Suppose that $p\ge 0$ and $\delta>0$. Let $\text{diam}(V)$ denote the diameter of a set $V$. For a subset $F$ of $\mathbb{R}^n$, a $\delta$-cover of $F$ is a countable collection of subsets $\{B_i\}$ such that $\mathrm{diam}(B_i)\le\delta$ and $F\subseteq\bigcup_{i}B_i$. Let
$$H^p_\delta(F)=\inf\left\{\sum_i(\mathrm{diam}(B_i)^p: \{B_i\} \text{ is a }\delta\text{-cover of }F\right\}.$$
The limit 
$$H^p(F)=\lim_{\delta\rightarrow 0} H_\delta^p(F)$$ is the $p$-dimensional Hausdorff measure of $F$. The Hausdorff dimension of $F$ is defined by
$$\dim_H(F)=\sup\{p\ge 0: H^p(F)=\infty\}=\inf\{p\ge 0:H^p(F)=0\}.$$
In the following, we introduce iterated function systems, self-similar measures, and self-similar sets. For our purpose, we only consider the case when the ambient space is $\mathbb{R}$ and state the related properties of iterated function systems, self-similar sets, and self-similar measures without proofs. The proofs of these properties, as well as a discussion of the more general case when the ambient space is $\mathbb{R}^n$, can be found in \cite{Hut}, for example.
 
An \textit{iterated function system} (IFS) is a finite collection of functions mapping from $\mathbb{R}$ to $\mathbb{R}$ of the form $f(x)=cx+b$ with $c\in(0,1)$ and $b\in\mathbb{R}$, where the constant $c$ is called the contraction ratio of $f$. Let $\mathcal{F}=\{f_i(x)=c_ix+b_i:1\le i\le l\}$ be an IFS. There exists a unique compact subset $K\subset\mathbb{R}$ invariant under $\mathcal{F}$ in the following sense:
\begin{align}\label{K eq}
K=\bigcup_{i=1}^l f_i(K).
\end{align}
Let $\mathcal{M}^1(\mathbb{R})$ be the space of compactly supported Borel probability measures on $\mathbb{R}$. Let $\boldsymbol\rho=(\rho_1,\rho_2,\dots,\rho_l)$ be a probability vector, i.e., $0<\rho_i<1$ for $1\le i\le l$ and $\sum_{i=1}^l\rho_i=1$. We define the operator $(\mathcal{F},\boldsymbol\rho)^P$ associated to $\mathcal{F}$ and $\rho$ on $\mathcal{M}^1(\mathbb{R})$ by 
$$(\mathcal{F},\boldsymbol\rho)^P\nu=\sum_{i=1}^l\rho_i(f_i)_*\nu,$$
where $(f_i)_*\nu$ is the pushforward of $\nu$ by $f_i$. Then, there exists a unique $\nu\in\mathcal{M}^1(\mathbb{R})$ invariant under $(\mathcal{F},\boldsymbol\rho)^P$, and moreover, the support of $\nu$ is the compact set $K$ invariant under $\mathcal{F}$. The set $K$ and the measure $\nu$ are called \textit{self-similar set} and \textit{self-similar measure}, respectively.

We say $\mathcal{F}$ satisfies \textit{the open set condition} if there exists a non-empty bounded open set $U$ such that
\begin{align}\label{OSC}
\bigcup_{i=1}^l f_i(U)\subset U\,\,\text{ and }\,\, f_i(U)\cap f_j(U)=\emptyset\,\,\text{ for }\,\,i\ne j.
\end{align}
With this condition assumed, the Hausdorff dimension of $K$ is the constant $s$ such that $\sum_{i=1}^l c_i^s=1$. In this case, $(c_1^s,c_2^s,\dots,c_l^s)$ is a probability vector, which we denote by $\boldsymbol{\rho}_\mathcal{F}$. Throughout this section, we assume that $\mathcal{F}$ satisfies the open set condition. 

Let $B(x,r)$ denote the 1-dimensional open ball centered at $x$ with radius $r$. Let $\mu$ be the self-similar measure invariant under $(\mathcal{F},\boldsymbol\rho_\mathcal{F})^P$, i.e., 
\begin{align}\label{mu eq}
\sum_{i=1}^l c_i^s(f_i)_*\mu=\mu.
\end{align}
 Then, $\mu$ has the following regularity property: the $\mu$-measure of of a ball $B(x,r)$ with $x\in\text{supp}(\mu)$ and $r\le \text{diam}(\text{supp}(\mu))$ has size about $r^s$, specifically, there exist positive real numbers $a_1$ and $a_2$ such that 
\begin{align}
a_1r^s\le\mu(B(x,r))\le a_2r^s.
\end{align}
We call $a_1$ and $a_2$ the regularity constants of $\mu$. 

For $k\in\mathbb{N}$, define $\mathcal{I}^k=\{(i_1,i_2,\dots, i_k)|i_j\in\{1,2,\cdots,l\}, 1\le j\le k\}$ as the collection of $k$-tuples with entries being integers between 1 and $l$.
%, and let 
%$$\mathcal{I}^*=\bigcup_{k=1}^\infty \mathcal{I}^k.$$ 
For $\alpha=(i_1,i_2,\dots, i_k)\in \mathcal{I}^k$, define $f_\alpha:\mathbb{R}\rightarrow\mathbb{R}$ by
$$f_\alpha: = f_{i_1}\circ f_{i_2}\circ \cdots\circ f_{i_k}.$$
Then $f_\alpha=c_\alpha x+b_\alpha$, where $c_\alpha=\prod_{j=1}^kc_{i_j}$ and $b_\alpha=f_\alpha(0).$ With this notation, for any $k\in\mathbb{N}$, by iterating the equality $K=\bigcup_{i=1}^l f_i(K)$ $k$ times, we have
\begin{align}\label{K union}
K=\bigcup_{\alpha\in \mathcal{I}^k}f_\alpha(K).
\end{align}
For $k=0$, we adapt the notations $\mathcal{I}^0=\{\emptyset\}$ and $f_\emptyset(x)=c_\emptyset x+b_\emptyset$, where $c_\emptyset=1$ and $b_\emptyset=0$. Let $\mathcal{I}^*=\bigcup_{k=0}^\infty\mathcal{I}^k.$
We define a partial order on $\mathcal{I}^*$: for $\alpha,\beta\in\mathcal{I}^*$, define
$$\alpha\prec\beta \quad\text{ iff }\quad\exists\gamma\in\mathcal{I}^* \text{ s.t. }\beta\text{ is the concatenation of the tuples }\alpha\text{ and }\gamma,\text{ i.e., }\beta=\alpha\gamma.$$
Define 
$$\Omega_{l}=\{\omega=(\omega_1,\omega_2,\dots)|\omega_i\in\{1,2,\dots, l\}\text{ for }i\in\mathbb{N}\}$$ as the collection of all the one-sided sequences of $l$ symbols. For $\omega=(\omega_1,\omega_2,\dots)\in\Omega_l$ and $k\in\mathbb{N}$, let $[\omega]_k=(\omega_1,\omega_2,\dots ,\omega_k)$ be the $k$-tuple obtained by truncating $\omega$ right after its $k$-th entry. For $k=0$, we adapt the notation $[\omega]_0=\emptyset$. We have the following lemma from \cite{Hut}.
\begin{lemma}[{\cite[Theorem 3.1]{Hut}}]\label{symbols}
 For every $\omega\in\Omega_l$, $\bigcap_{k=1}^\infty f_{[\omega]_k}(K)$ is a point in $K$. The map $\pi:\Omega_l\rightarrow K$ defined by $\pi(\omega)=\bigcap_{k=1}^\infty f_{[\omega]_k}(K)$ is surjective. 
\end{lemma}
\subsection{Branches of \texorpdfstring{$\mu$}{mu}.}\label{S 2.2}
In this subsection, we follow the same setting as in \S \ref{S 2.1}: $\mathcal{F}=\{f_i(x)=c_ix+b_i:1\le i\le l\}$ is an IFS satisfying the open set condition, $K$ is the unique compact set invariant under $\mathcal{F}$, and $\mu$ is the unique compactly supported Borel probability measure invariant under $(\mathcal{F},\boldsymbol\rho_\mathcal{F})^P$, and $s$ is the Hausdorff dimension of $K$. We assume that $c_1$ is the smallest contraction ratio among the contraction ratios of $f_i$'s in $\mathcal{F}$. By a branch of $\mu$, we mean the pushforward of $\mu$ by $f_\alpha$ for some $\alpha\in\mathcal{I}^*$, and we use the following notation
$$\mu_\alpha:=(f_\alpha)_*\mu.$$
The branch $\mu_\alpha$ is again a self-similar measure by the following lemma.
\begin{lemma}\label{branch}
Let $\mu_\alpha$ be a branch of $\mu$. Then,  $\mu_\alpha$ is the unique compactly supported Borel probability measure invariant under the operator $(\mathcal{F}_\alpha,\boldsymbol\rho_{\mathcal{F}_\alpha})^P$, where 
$$\mathcal{F}_\alpha=\{ f_\alpha\circ f_i\circ f_\alpha^{-1}:1\le i\le l\}.$$
The unique compact set invariant under $\mathcal{F}_\alpha$ is $f_\alpha(K)$, which is the support of $\mu_\alpha$ and has Hausdorff dimension $s$. 
\end{lemma}
\begin{proof}
Let $U$ be an open set so that \eqref{OSC} holds. Since $f_\alpha$ is a linear map, it maps disjoint sets to disjoint sets, so $f_\alpha\circ f_i (U)$'s are disjoint. Then, we have
\begin{align}
\bigcup_{i=1}^l f_\alpha\circ f_i\circ f_\alpha^{-1}(f_\alpha (U))=\bigcup_{i=1}^l f_\alpha\circ f_i(U)=f_\alpha\left(\bigcup_{i=1}^l f_i(U)\right)\subset f_\alpha (U).
\end{align}
Therefore, $\mathcal{F}_\alpha$ satisfies the open set condition with the open set being $f_\alpha(U)$. 

Note that the contraction ratio of $f_\alpha\circ f_i\circ f_\alpha^{-1}$ is again $c_i$, so $\boldsymbol\rho_{\mathcal{F}_\alpha}=\boldsymbol\rho_\mathcal{F}$. By direct checking using \eqref{K eq} and \eqref{mu eq}, $f_\alpha(K)$ satisfies 
$\bigcup_{i=1}^l f_\alpha\circ f_i\circ f_\alpha^{-1}(f_\alpha(K))=f_\alpha(K),$
and $\mu_\alpha$ satisfies
$\sum_{i=1}^l c_i^s(f_\alpha\circ f_i\circ f_\alpha^{-1})_*\mu_\alpha=\mu_\alpha.$ Thus, $f_\alpha(K)$ is the unique compact set invariant under $\mathcal{F}_\alpha$ and $\mu_\alpha$ is the unique measure in $\mathcal{M}^1(\mathbb{R})$ invariant under $(\mathcal{F}_\alpha,\boldsymbol\rho_{\mathcal{F}_\alpha})^P$. By our discussion in \S \ref{S 2.1}, the support of a self-similar set associated with an IFS is the self-similar set invariant under the IFS, so $f_\alpha(K)$ is the support of $\mu_\alpha$. Since $\boldsymbol\rho_{\mathcal{F}_\alpha}=\boldsymbol\rho_\mathcal{F}$, the Hausdorff dimension of $f_\alpha(K)$ remains the same as that of $K$, which is $s$. 
\end{proof}

Since the support of $\mu_\alpha$ is $f_\alpha(K)$, by \eqref{K union}, $\text{supp}(\mu_\alpha)$ is a contained in $\text{supp}(\mu)=K$. By the above lemma, the notion of branch can be applied to $\mu_\alpha$. By definition, a branch of $\mu_\alpha$ is of the form
\begin{align}\label{subbranch}
\left((f_\alpha\circ f_{i_1}\circ f_\alpha^{-1})\circ(f_\alpha\circ f_{i_2}\circ f_\alpha^{-1})\circ\cdots\circ (f_\alpha\circ f_{i_n}\circ f_\alpha^{-1})\right)_*\mu_\alpha,
\end{align}
where $n$ is some non-negative integer and $1\le i_1,i_2,\dots,i_n\le l$. Let $\gamma=(i_1,i_2,\dots, i_n)$, then \eqref{subbranch} can be re-expressed as 
$$(f_\alpha\circ f_\gamma\circ f_\alpha^{-1})_*(f_\alpha)_*\mu=(f_\alpha\circ f_\gamma)_*\mu=(f_{\alpha\gamma})_*\mu.$$
Thus, a branch of $\mu_\alpha$ is a branch of $\mu$. 

Conversely, suppose $\beta$ is an element in $\mathcal{I}^*$ such that $\alpha\prec\beta$. Let $\gamma\in\mathcal{I}^*$ be such that $\beta=\alpha\gamma.$ Then, we have 
 $$(f_{\alpha\gamma})_*\mu=(f_\alpha\circ f_\gamma)_*\mu=(f_\alpha\circ f_\gamma\circ f_\alpha^{-1})_*(f_\alpha)_*\mu,$$
 hence, $\mu_\beta$ is a branch of $\mu_\alpha$.
 
The regularity constants of $\mu_\alpha$ are related to the regularity constants of $\mu$ by the following lemma.

\begin{lemma}\label{regularity lemma}
Let $\mu_\alpha$ be a branch of $\mu$. Then, for any ball $B(x,r)$ centered at $x\in\textup{supp}(\mu_\alpha)$ with radius $r\le \textup{diam}(\textup{supp}(\mu_\alpha))$, the following $\mu_\alpha$-measure estimate for $B(x,r)$ holds:
 \begin{align*}
 a_1c_\alpha^{-s}r^s\le \mu_\alpha(B(x,r))\le a_2c_\alpha^{-s}r^s.
 \end{align*}

\end{lemma}
\begin{proof}
Let $B(x,r)$ be a ball satisfying the condition in the statement. Since $f_\alpha$ is a linear map with contraction ratio $c_\alpha$, $f_\alpha^{-1}(B(x,r))=B(f_\alpha^{-1}(x),c_\alpha^{-1}r)$. In view of the fact that $\text{supp}(\mu_\alpha)=f_\alpha(\text{supp}(\mu))$, we have 
$$f_\alpha^{-1}(x)\in \text{supp}(\mu)\,\text{ and }\,c^{-1}_\alpha r\le \text{diam}(\text{supp}(\mu)).$$
It follows by the regularity property of $\mu$ that
\begin{align*}
 a_1(c_\alpha^{-1}r)^s\le\mu\left(f_\alpha^{-1}(B(x,r))\right)\le a_2(c_\alpha^{-1}r)^s.
 \end{align*}
By definition, $\mu_\alpha(B(x,r))=\mu\left(f_\alpha^{-1}(B(x,r))\right)$, so we obtain
$$a_1c_\alpha^{-s}r^s\le\mu_\alpha(B(x,r))\le a_2c_\alpha^{-s} r^s.$$
\end{proof}

We need one more lemma on finding a proper branch $\mu_\beta$ whose support contains a given point in $\text{supp}(\mu_\alpha)$.
\begin{lemma}\label{lemma 1.4}
 Let $x$ be a point in $\textup{supp}(\mu_\alpha)$, where $\alpha\in\mathcal{I}^*$, and let $Q$ be a positive constant no greater than $\textup{diam}(\textup{supp}(\mu_\alpha))/c_1$. Then, there exists a branch $\mu_{\beta}$ of $\mu$ such that $\alpha\prec\beta$, $x\in \textup{supp}(\mu_\beta)$, and $\textup{diam}(\textup{supp}(\mu_\beta))\in[c_1Q,Q]$.
 
\end{lemma}
\begin{proof}
First, we prove the lemma assuming $\alpha=\emptyset$. By Lemma \ref{symbols}, the map $\pi:\Omega_l\rightarrow K$ is surjective, so $\pi^{-1}(x)$ is not empty. Let $\omega=(\omega_1,\omega_2,\dots)\in\Omega_l$ be a preimage of $x$ under $\pi$. We truncate $\omega$ after the $n$-th entry for which $c_{[\omega]_n}\cdot\text{diam}(\text{supp}(\mu))\in[c_1Q,Q]$. The existence of such $n$ is guaranteed by the assumption on $Q$ and the fact that $\lim_{n\rightarrow\infty}c_{[w]_n}=0$. Take $\beta=[\omega]_n$, then $\alpha=\emptyset\prec\beta$. By the definition of $\pi$, $x=\pi(\omega)=\bigcap_{j=1}^\infty f_{[\omega]_j}(K)$, so $x\in f_\beta(K)=\text{supp}(\mu_\beta)$. Since the contraction ratio of $f_\beta$ is $c_\beta$, we have $\text{diam}(\text{supp}(\mu_\beta))=c_{\beta}\cdot\text{diam}(\text{supp}(\mu))\in [c_1Q,Q]$.

Now we consider the case when $\alpha$ is a general element of $\mathcal{I}^*$. By Lemma \ref{branch}, the contraction ratios of functions in $\mathcal{F}_\alpha$ are the same as the contraction ratios of functions in $\mathcal{F}$. Thus, the above argument goes through with $\mu$ replaced by $\mu_\alpha$ and implies that there exists a branch of $\mu_\alpha$, which we denote by $\nu$, such that $x\in\text{supp}(\nu)$ and $\text{diam}(\text{supp}(\nu))\in[c_1Q,Q]$. By the discussion following Lemma \ref{branch}, $\nu$ is indeed some branch $\mu_\beta$ of $\mu$ with $\alpha\prec\beta$.  
\end{proof}

\subsection{Effective equidistribution of self-similar measures}\label{equi}
One of the main ingredients in our proofs of Theorems \ref{thm} and \ref{main thm} is the effective equidistribution of self-similar measures along the horocyclic trajectories under the pushforward by geodesic flow in the homogeneous space $\text{SL}_2(\mathbb{R})/\text{SL}_2(\mathbb{Z})$, which has been proved in \cite{BHZ} and \cite{KL}. 

We follow the same notations for iterated function systems and the associated self-similar measures as in \S \ref{S 2.2}: $\mathcal{F}$ is an iterated function system satisfying the open set condition, and $\mu$ is the self-similar measure invariant under the operator $(\mathcal{F},\rho_\mathcal{F})^P$. Now, we introduce the dynamical setting. Let $G=\text{SL}_2(\mathbb{R})$, $\mathfrak{g}=\text{Lie}(G)$, the Lie algebra of $G$, and $X=\text{SL}_2(\mathbb{R})/\text{SL}_2(\mathbb{Z})$. Then, $\mathfrak{g}$ has a standard basis consisting of the following elements:
$$e_{-1}=
\begin{pmatrix}
0&0\\
1&0
\end{pmatrix},\quad
e_0=
\begin{pmatrix}
1&0\\
0&-1
\end{pmatrix},\quad
e_1=
\begin{pmatrix}
0&1\\
0&0
\end{pmatrix}.$$
We equip $G$ with a right-invariant Riemannian metric such that the basis $\{e_{-1},e_0,e_1\}$ is orthonormal. Denote by $D_{-1}, D_0,$ and $D_1$ the differential operators on $X$ induced by $e_{-1}, e_0,$ and $e_1$, respectively. Denote by $C^\infty(X)$ the set of smooth functions on $X$. For $f\in C^\infty(X)$, define the $S_{\infty,1}$ Sobolev norm of $f$ by 
\begin{align*}
\mathcal{S}_{\infty, 1}(f)=\|f\|_\infty + \|D_{-1}f\|_\infty+\|D_0f\|_\infty+\|D_1f\|_\infty,
\end{align*}
where $\|\cdot\|_\infty$ is the $L^\infty$ norm on $C^\infty(X)$, and let
\begin{align*}
B^\infty_{\infty,1}(X)=\{f\in C^\infty(X):\mathcal{S}_{\infty,1}(f)<\infty\}.
\end{align*}
Let $m_X$ be the $\text{SL}_2(\mathbb{R})$-invariant Haar measure on $X$. For $t,s\in\mathbb{R}$, let
$$
g(t)=
\begin{pmatrix}
e^t&0\\
0&e^{-t}
\end{pmatrix}
\text{ and }
u(s)=
\begin{pmatrix}
1&s\\
0&1
\end{pmatrix}.
$$
 
Let $\mathrm{exp}(\cdot)$ denote the exponential map from $\mathfrak{g}$ to $G$. We equip the vector space $\mathfrak{g}$ with a norm. Denote by $B_{r}^\mathfrak{g}(0)$ the ball in $\mathfrak{g}$ centered at 0 with radius $r$, and let $r_0$ be such that $\mathrm{exp}: B^\mathfrak{g}_{r_0}(0)\rightarrow \mathrm{exp}(B_{r_0}^\mathfrak{g}(0))$ is a diffeomorphism.  For $x\in X$, we define the injectivity radius at $x$ as \footnote{We note that this definition of injectivity radius is different from that in \cite{BHZ}, but this will not affect the final result.}
 \begin{align*}
 \mathrm{inj}(x)=\sup\{ r\le r_0: \text{ the map } g\mapsto gx \text{ is injective on }\mathrm{exp}(B_r^\mathfrak{g}(0) )\}.
 \end{align*}
Next, we recall the following key equidistribution result from~\cite {BHZ}, which generalizes earlier special cases in~\cite{KL}.
\begin{theorem}[{\cite[Theorem B]{BHZ} and \cite[Corollary 6.4]{KL}}]
\label{thm:BHZ}
There exists a constant $\kappa>0$ such that for all $t>0$, $x\in X$, and $\varphi\in B^\infty_{\infty,1}(X)$, the following holds
$$\int_\mathbb{R}\varphi(g(t)u(s)x)\,d\mu(s)=\int_{X}\varphi\, dm_{X}+O\left(\textup{inj}(x)^{-1}\mathcal{S}_{\infty,1}(\varphi)e^{-\kappa t}\right),$$
where $\textup{inj}(x)$ is the injectivity radius at $x$ and the implicit constant in $O(\cdot)$ only depends on the IFS $\mathcal{F}$.
\end{theorem}

As a corollary, we can derive the effective equidistribution of branches of $\mu$ from the above theorem. 

\begin{corollary}\label{corollary 3.2}
Let $\varphi\in B_{\infty, 1}^\infty(X)$. Then, for any $\theta\in\mathcal{I}^*$, we have
\begin{align*}
\int \varphi(g(t)u(s)x)\,d\mu_\theta(s)=\int_X\varphi \,d m_X+O\left( \textup{inj}(x)^{-1}c_\theta^{-1-\frac{\kappa}{2}}\mathcal{S}_{\infty,1}(\varphi)e^{-\kappa t}\right),
\end{align*}
where $\kappa$ is as in Theorem~\ref{thm:BHZ} and the implicit constant in $O(\cdot)$ only depends on the IFS $\mathcal{F}$.
\end{corollary}
\begin{proof}
Since $\mu_\theta=(f_\theta)_*\mu$ and $f_\theta(s)=c_\theta s+b_\theta$, we have
\begin{align*}
\int \varphi(g(t)u(s) x)\,d\mu_\theta(s)&=\int \varphi(g(t) u(c_\theta s+b_\theta)x)\,d\mu(s)\\
&=\int\varphi\left(g\left(t+\ln \sqrt{c_\theta}\right)u(s)g\left(-\ln \sqrt{c_\theta}\right)u(b_\theta)x\right)\, d\mu(s).
\end{align*}
Therefore, by Theorem \ref{thm:BHZ}, 
\begin{align*}
\int \varphi(g(t)u(s) x)\,d\mu_\theta(s)&=\int \varphi\,dm_X+O\left(\text{inj}\left( g(-\ln\sqrt{ c_\theta})u(b_\theta)x\right)^{-1}\mathcal{S}_{\infty,1}(\varphi)e^{-\kappa \left(t+\ln \sqrt{c_\theta}\right)}\right).
\end{align*}
 Note that there exists a constant $C_1>0$ such that, for any $h\in G$ and any $y\in X$, $\mathrm{inj}(hy)\ge C_1\cdot\|\mathrm{Ad}(h^{-1})\| \mathrm{inj}(y)$, where $\mathrm{Ad}$ is the adjoint action of $G$ on $\mathfrak{g}$ and $\|\cdot\|$ is the operator norm. Therefore, we have
$$\mathrm{inj}(g(-\ln \sqrt{c_\theta})u(b_\theta)x)\ge C_1\|\mathrm{Ad}(g(\ln \sqrt{c_\theta}))\|\mathrm{inj}(u(b_\theta)x)\ge C_1\cdot c_\theta^{-1}\mathrm{inj}(u(b_\theta)x).$$
Since $u(b_\theta)$ is uniformly bounded over all $\theta\in \mathcal{I}^*$, there exist constants $C_2>0$ such that $\textup{inj}\left(u(b_\theta)x\right) \ge C_2\cdot \textup{inj}(x).$ Thus, we have
 $$\text{inj}\left( g(-\ln\sqrt{ c_\theta})u(b_\theta)x\right)^{-1}\le C_1^{-1}C_2^{-1} c_\theta^{-1}\mathrm{inj}(x)^{-1}.$$
The desired effective equidistribution result then follows.
\end{proof}

Now we introduce two series of sets, whose limit sets will be used to approximate $\psi$-well approximable sets in Sections 3 and 4.

Let $\psi:\mathbb{N}\rightarrow\mathbb{R}$ be a monotonically non-increasing function, and let $P(\mathbb{Z}^{2})=\{(p,q)\in\mathbb{Z}^2: \mathrm{gcd}(p,q)=1\}$ be the collection of primitive vectors in $\mathbb{Z}^{2}$. For $m\in\mathbb{N}$, define $A_m(\psi)$ and $D_m(\psi)$ by
$$A_m(\psi)=\big\{x\in\mathbb{R}: \exists (p,q)\in P(\mathbb{Z}^{2})\text{ so that } 2^{m}\le q<2^{m+1}\,\text{ and }\,|qx-p|<\psi(2^m)\big\},$$
and
$$D_m(\psi)=\{x\in\mathbb{R}: \exists (p,q)\in P(\mathbb{Z}^{2})\text{ so that }2^{m-1}\le q<2^m\text{ and } |qx-p|<\psi(2^m)\}.$$
Furthermore, we define $\mathcal{A}_m(\psi)$ and $\mathcal{D}_m(\psi)$ by
 $$\mathcal{A}_m(\psi)=\Big\{B\Big(\frac{p}{q},\frac{\psi(2^m)}{q}\Big): (p,q)\in P(\mathbb{Z}^{2})\text{ and }2^m\le q<2^{m+1}\Big\},$$
 and
 $$\mathcal{D}_m(\psi)=\Big\{B\Big(\frac{p}{q},\frac{\psi(2^m)}{q}\Big):(p,q)\in P(\mathbb{Z}^2)\text{ and }2^{m-1}\le q<2^m\Big\}.$$
Then, it is clear that $A_m(\psi)=\bigcup_{B\in \mathcal{A}_m(\psi)}B$ and $D_m(\psi)=\bigcup_{B\in\mathcal{D}_m(\psi)}B$.

\begin{proposition}[{\cite[Theorem 9.1]{KL}}]
\label{A_n}
Let $\psi$ be defined as above, and let $\kappa$ be as in Theorem \ref{thm:BHZ}. Then, there exists a constant $\hat C$ depending on $\mathcal{F}$ and $\psi$ so that
\begin{align*}
\mu\left(A_n(\psi)\right)\le \hat C \left(2^n\psi(2^n)+2^{-\frac{\kappa }{2} n}\right).
\end{align*}

\end{proposition}

Corollary \ref{corollary 3.2} together with Lemma 12.7 in \cite{KL} gives us an estimate of $\mu_\theta(D_m(\psi))$. 

\begin{proposition}
\label{proposition 3.3}
There exist positive constants $C$, $C'$, and $C_\mathcal{F}$ so that for any $\theta\in\mathcal{I}^*$
\begin{align*}
 \mu_\theta(D_m(\psi))\ge C\cdot2^m\psi(2^m)-C'(2^m\psi(2^m))^2-c_\theta^{-1-\frac{\kappa}{2}}C_{\mathcal{F}}\left(2^{-m}\psi(2^m)\right)^{\frac{\kappa}{2}},
 \end{align*}
 and
 \begin{align*}
 \mu_\theta(D_m(\psi))\le 12C\cdot2^m\psi(2^m)+c_\theta^{-1-\frac{\kappa}{2}}C_{\mathcal{F}} \left(2^{-m}\psi(2^m)\right)^{\frac{\kappa}{2}}.
\end{align*}
Here, $C$ and $C'$ are some fixed constants independent of the IFS $\mathcal{F}$, and $C_\mathcal{F}$ depends only on $\mathcal{F}$. 
\end{proposition}
\begin{proof}
    It was shown in~\cite[Lemma 12.7]{KL} that the effective equidistribution result for $\mu$ in Theorem~\ref{thm:BHZ} implies that the assertion of the proposition holds for $\mu$.
    Indeed, to convert from the form of the estimate in \textit{loc. cit.} to the form appearing in our statement, one uses~\cite[Lemma 8.1]{KL}.
    For general branches $\mu_\theta$, the statement follows by the same argument in the proof of~\cite[Lemma 12.7]{KL} by using Corollary~\ref{corollary 3.2} for the branch measures $\mu_\theta$ in place of Theorem~\ref{thm:BHZ}.
\end{proof}

\subsection{Notations.} We use notation $A\gtrsim B$ (resp. $A\lesssim B$) to denote that $A\ge cB$ (resp. $A\le cB$) for some positive constant $c$, and we use the notation $A\asymp B$ to denote that $c^{-1}B\le A\le cB$ for some positive constant $c$. %We use 

\section{An Upper bound for the Hausdorff dimension of $\mathbf{ W(\psi)\cap K}$}\label{S 3}
 
The goal of this section is to prove Theorem \ref{thm}. Throughout, we fix an IFS $\mathcal{F}=\{f_i(x)=c_ix+b_i:1\le i\le l\}$ satisfying the open set condition. $K$ is the self-similar set invariant under $\mathcal{F}$, $\mu$ is the self-similar measure invariant under $(\mathcal{F},\boldsymbol\rho_\mathcal{F})^P$, and $s$ is the Hausdorff dimension of $K$. 

\begin{lemma}\label{lemma2}
Suppose $\psi$ is a monotonically non-increasing function satisfying $\sum_{q=1}^\infty\psi(q)<\infty$, then there exists an integer $M$, depending on $\psi$, such that the closures of any two distinct 1-dimensional balls in $\mathcal{A}_m(\psi)$ or $\mathcal{D}_m(\psi)$ are disjoint, respectively, for any $m\ge M$. 
\end{lemma}
\begin{proof}
Let $\xi$ be an arbitrary positive constant. We claim that $\psi(q)<\xi/q$ for large enough $q$. Suppose by contradiction that there exists an infinite sequence $\{q_i\}_{i=1}^\infty$ such that $\psi(q_i)\ge \xi/q_i$. By passing to a subsequence, we may assume $q_i>2^{i-1}q_{i-1}$. By the monotonicity of $\psi$, we have
\begin{align*}
\sum_{q=1}^{q_i} \psi(q)&\ge q_1\cdot\psi(q_1)+(q_2-q_1)\psi(q_2)+\cdots+(q_i-q_{i-1})\psi(q_i)\\
&\ge q_1\cdot\frac{\xi}{q_1}+(q_2-q_1)\cdot\frac{\xi}{q_2}+\cdots+(q_i-q_{i-1})\frac{\xi}{q_i}\ge i\xi-\xi\cdot\sum_{j=1}^{i-1}\frac{1}{2^j}.
\end{align*}
Taking the limit $i\rightarrow\infty$, we obtain $\sum_{q=1}^\infty\psi(q)=\infty$, which contradicts our assumption on $\psi$. By the above argument, there exists an integer $M$ such that $\psi(2^m)<1/8\cdot 2^m$ for any $m\ge M$. Let $m\ge M$ be a fixed integer, and suppose $B(p_j/q_j,\psi(2^m)/q_j), j=1,2$, are two distinct balls in $\mathcal{A}_m(\psi)$. Since $(p_1,q_1)$ and $(p_2,q_2)$ are primitive, $(p_1,q_1)\ne(p_2,q_2)$ implies $p_1q_2\ne p_2q_1$. Since $2^m\le q_j< 2^{m+1}$, we have
 \begin{align*}
\left|\frac{p_1}{q_1}-\frac{p_2}{q_2}\right|=\left|\frac{p_1q_2-p_2q_1}{q_1q_2}\right|\ge\frac{1}{q_1q_2}>\frac{1}{2^{2m+2}}>\frac{\psi(2^m)}{q_1}+\frac{\psi(2^m)}{q_2},
\end{align*}
thus the closures of the two balls are disjoint. Similarly, suppose $B(p_j/q_j,\psi(2^m)/q_j), j=1,2$, are two distinct balls in $\mathcal{D}_m(\psi)$, then $2^{m-1}\le q_j<2^m$ implies 
\begin{align*}
\left|\frac{p_1}{q_1}-\frac{p_2}{q_2}\right|\ge\frac{1}{q_1q_2}\ge\frac{1}{2^{2m}}>\frac{\psi(2^m)}{q_1}+\frac{\psi(2^m)}{q_2},
\end{align*}
so the closures of the two balls are disjoint.
\end{proof}

With the above lemma, we are ready to prove Theorem \ref{thm}. 
\begin{proof}[Proof of Theorem \ref{thm}]
In this proof, we consider two functions: $\psi$ and its constant multiple $2\psi$. Let $\mathcal{A}_m(2\psi)$, $\mathcal{A}_m(\psi)$, $A_m(2\psi)$, and $A_m(\psi)$ be defined as in \S 2.3.
Since $\sum_{q=1}^\infty\psi(q)<\infty$ and $\sum_{q=1}^\infty2\psi(q)<\infty$, by Lemma ~\ref{lemma2}, there exists an integer $M_0$ such that the balls in $\mathcal{A}_m(2\psi)$ and $\mathcal{A}_m(\psi)$ are disjoint, respectively, for any $m\ge M_0$. By increasing the value of $M_0$, we may assume that $\psi(2^{M_0})/2^{M_0}$ is less than $\text{diam}(K)$. Let $\mathcal{A}^*_m(\psi)$ be the collection of balls in $\mathcal{A}_m(\psi)$ that have non-empty intersection with $K$. Note that $\mathcal{A}_{m}(2\psi)$ consists of balls with the same centers as those in $\mathcal{A}_m(\psi)$, but with doubled radii. By doubling the radii of balls in $\mathcal{A}_m^*(\psi)$, we obtain a subcollection of balls in $\mathcal{A}_m(2\psi)$, which we denote by $\mathcal{A}_m^*(2\psi)$. 

Let $m\ge M_0$ be a fixed integer, and let $B(x,r)$ be a ball in $\mathcal{A}_m^*(\psi)$. Since $\psi$ is a non-increasing function, by the definition of $\mathcal{A}_m(\psi)$, we have
$$\frac{\psi(2^m)}{2^{m+1}}<r\le\frac{\psi(2^m)}{2^m}\le\frac{\psi(2^{M_0})}{2^{M_0}}\le\text{diam}(K).$$
Take a point $y$ in $B(x,r)\cap K$. Then, $B(y,r)$ is contained in $B(x,2r)$, and so we have
$$\mu(B(x,2r))\ge\mu(B(y,r))\ge a_1r^s\ge a_12^{-(m+1)s}\psi^s(2^m),$$
where the second inequality is due to the regularity property of $\mu$. As $B(x,r)$ varies over all the balls in $\mathcal{A}_m^*(\psi)$, $B(x,2r)$ covers every element in $\mathcal{A}_m^*(2\psi)$. Thus, any ball in $\mathcal{A}_m^*(2\psi)$ has measure at least $a_12^{-(m+1)s}\psi^s(2^m)$. Since $\mathcal{A}_m(2\psi)$ consists of disjoint balls and $\mathcal{A}_m^*(2\psi)$ is a subset of $\mathcal{A}_m^*(2\psi)$, we have
\begin{align}\label{card A_m}
 \mu(A_m(2\psi))\ge \sum_{B\in\mathcal{A}^*_m(2\psi)}\mu(B)\ge \#\mathcal{A}^*_m(2\psi)\cdot a_12^{-(m+1)s}\psi^s(2^m),
\end{align}
where $\#\mathcal{A}^*_m(2\psi)$ is the cardinality of the set $\mathcal{A}^*_m(2\psi)$. By Proposition \ref{A_n}, there exists a constant $\kappa>0$ and a constant $\hat C>0$, depending on $\psi$ and $\mu$, such that
\begin{align}\label{measure estimate}
\mu(A_m(2\psi))\le \hat C\left( 2^m\cdot 2\psi(2^m)+2^{-\frac{\kappa}{2}m}\right).
\end{align}
Now we choose $J=\kappa/2$. Since $\psi(2^m)\ge 2^{-\left(\kappa/2+1\right)m},$ then the error term $2^{-\kappa m/2}$ in the above estimate can be majorized by $2^m\psi(2^m)$, so we have
\begin{align*}
\mu(A_m(2\psi))\le 3\hat C\cdot2^m\psi(2^m).
\end{align*}
The sets $\mathcal{A}_m^*(\psi)$ and $\mathcal{A}_m^*(2\psi)$ have the same cardinalities. By combining \eqref{card A_m} and \eqref{measure estimate}, we obtain the following estimate of the cardinality of $\mathcal{A}_m^*(\psi)$:
 \begin{align*}
\#\mathcal{A}_m^*(\psi)= \#\mathcal{A}^*_m(2\psi)\le 3\hat Ca_1^{-1}\cdot 2^{m(1+s)+s}\psi^{1-s}(2^m).
 \end{align*}
 
 We note that $W(\psi)$ is a subset of $\limsup A_m(\psi)$, so an upper bound for $\dim_H(\limsup A_m(\psi)\cap K)$ is an upper bound for $\dim_H(W(\psi)\cap K)$. 
Indeed, suppose $y$ is a point in $ W(\psi)$. Then, by definition, there are infinitely many integers $m>0$ and pairs $(p,q)\in\mathbb{Z}\times \mathbb{N}$ with $2^m\le q<2^{m+1}$ such that 
\begin{align}\label{ypq}
\left|y-\frac{p}{q}\right|<\frac{\psi(q)}{q}.
\end{align} 
Among these pairs of $(p,q)$, the scalar multiples of a fixed pair $(p,q)$ can appear only finitely many times since $|y-np/nq|=|y-p/q|$ is fixed and $\psi(nq)/nq\rightarrow 0$ as $n\rightarrow\infty$. Therefore, there are indeed infinitely many primitive pairs of $(p,q)$ such that \eqref{ypq} holds, which implies that $y$ is in $\limsup A_m(\psi)$. Hence, $W(\psi)$ is contained in $\limsup A_m(\psi)$. 

Next, we give an upper bound for the Hausdorff dimension of $\limsup A_m(\psi)\cap K$. 
  
Let $\delta>0$ be given. There exists an integer $M_\delta>M_0$ such that $2\psi(2^m)/2^m<\delta$ for any $m\ge M_\delta$. Then, $\bigcup_{m=M_\delta}\mathcal{A}^*_m(\psi)$ is a cover of $\big(\bigcup_{m=M_\delta}^\infty A_m(\psi)\big)\cap K$ consisting of balls with diameters less than $\delta$. Since $\limsup A_m(\psi)\cap K$ is a subset of $\left(\bigcup_{m=M_\delta}^\infty A_m(\psi)\right)\cap K$, $\bigcup_{m=M_\delta}\mathcal{A}^*_m(\psi)$ is also a $\delta$-cover of $\limsup A_m(\psi)\cap K$, so for $0<l\le s$, we have
 \begin{align}
 H_\delta^l\left(\limsup A_m(\psi)\cap K\right)
 &\le \sum_{m=M_\delta}^\infty\#\mathcal{A}^*_m(\psi)\cdot\left(2\cdot\frac{\psi(2^m)}{2^m}\right)^l\notag\\
 &\le\sum_{m=M_\delta}^\infty 12a_1^{-1}\hat C\cdot 2^{m(1+s-l)}\psi^{1+l-s}(2^m).\label{upper bound}
 \end{align} 
Suppose $\sum_{q=1}^\infty q^{s-l}\psi^{1+l-s}(q)<\infty$. Since $\psi$ is non-increasing, we have 
\begin{align*}
\sum_{m=1}^\infty 2^{m(1+s-l)}\psi^{1+l-s}(2^m)&\le\sum_{m=1}^\infty\sum_{q=2^{m-1}+1}^{2^m}2\cdot (2q)^{s-l}\psi^{1+l-s}(q)\\
&\le2^{1+s-l}\sum_{q=1}^\infty q^{s-l}\psi^{1+l-s}(q)<\infty.
\end{align*}
As $\delta$ goes to zero, $M_\delta$ tends to $\infty$. Taking the limit $\delta\rightarrow 0$ in inequality \eqref{upper bound}, we obtain 
 \begin{align*}
 H^l(\limsup A_m(\psi)\cap K)\le \lim_{M_\delta\rightarrow\infty}\sum_{m=M_\delta}^\infty 12a_1^{-1}\hat C\cdot 2^{m(1+s-l)}\psi^{1+l-s}(2^m)=0.
 \end{align*}
 Therefore, $H^l(W(\psi)\cap K)=0$. It then follows by the definition of Hausdorff dimension that \eqref{thm 1.1} holds.

For $1<v<\kappa/ 2+1$, $\psi_v$ satisfies $\psi_v(2^n)\ge 2^{\left(-\kappa/2+1\right)n}.$
For any $l>s-1+2/(v+1)$, $-(v+1)l+(v+1)s-v<-1$, so we have
  \begin{align*}
 \sum_{q=1}^\infty q^{s-l}\psi_v^{1+l-s}(q)=\sum_{q=1}^\infty q^{-(v+1)l+(v+1)s-v}<\infty.
 \end{align*}
 Hence, $\dim_H(W(v)\cap K)\le s-1+2/(v+1)$ by \eqref{thm 1.1}.
 \end{proof}
 
 \begin{remark}
 We note that the method used in the proof of Theorem \ref{thm} cannot give a good upper bound for $\dim_H\left(W(\psi)\cap K\right)$ when $\psi$ decreases too fast. Specifically, if $\psi(2^n)<2^{-(\kappa/2+1)n}$, then the measure estimate of $A_m(2\psi)$ in \eqref{measure estimate} is majorized by the error term $2^{-\kappa m/2}$. In this case, the faster $\psi$ decreases, the larger $l$ we need so that the resulting upper bound in \eqref{upper bound} converges; therefore, we would obtain an even higher upper bound for the Hausdorff dimension of $W(\psi)\cap K$ when $\psi$ decreases more rapidly.
 \end{remark}
 
 \section{A Lower bound for the Hausdorff dimension of $\mathbf{W(v)\cap K }$}\label{S 4}
 The goal of this section is to prove Theorem~\ref{main thm}. As in \S \ref{S 3}, throughout, we fix an IFS $\mathcal{F}=\{f_i(x)=c_ix+b_i:1\le i\le l\}$ satisfying the open set condition. $K$ is the self-similar set invariant under $\mathcal{F}$, $\mu$ is the self-similar measure invariant under $(\mathcal{F},\boldsymbol\rho_\mathcal{F})^P$, and $s$ is the Hausdorff dimension of $K$. We also assume that $c_1$ is the smallest contraction ratio among the contraction ratios of functions in $\mathcal{F}$.

 \begin{proposition}\label{main construction}
 Suppose $\mathcal{F}$ satisfies the open set condition. Let $\psi(q)=1/q^v$ be such that $1<v<\kappa(v+1)/4+1$, where $\kappa$ is the constant as in Theorem \ref{equi}, and let $\mathcal{D}_m(\psi)$ be the sets defined as in Section \ref{equi}. Then, there exists a geometric sequence $\{M_n\}_{n=0}^\infty$ whose common ratio is strictly greater than 1, and subsets $\mathcal{D}_{M_n}^*(\psi)$ of $\mathcal{D}_{M_n}(\psi)$ for $n=1,2,\dots$, such that the following statements hold for all integers $n\ge 1$.
 \begin{enumerate}[(1)]
 \item\label{A} The  closures of the 1-dimensional balls in $\mathcal{D}_{M_n}^*(\psi)$ are disjoint, and any ball $B(x,r)$ in $\mathcal{D}_{M_n}^*(\psi)$ satisfies
 \begin{align*}
 \mu(B(x,r))\ge\frac{2^sa_1}{3^s} r^s.
 \end{align*}
 \item\label{B} Every ball in $\mathcal{D}_{M_{n+1}}^*(\psi)$ is a subset of some ball in $\mathcal{D}_{M_{n}}^*(\psi)$. Let $B'$ be a ball in $\mathcal{D}_{M_{n}}^*(\psi)$. Then
$$
 \#\{B\in\mathcal{D}_{M_{n+1}}^*(\psi):B\subset B'\}
\asymp2^{M_{n+1}(v+1)s-M_{n}(v+1)s-M_{n+1}(v-1)},
$$
 where the implicit constant depends only on the IFS $\mathcal{F}$.
 \item\label{C} Let $E_n$ be the union of balls in $\mathcal{D}_{M_n}^*(\psi)$, i.e., $E_n=\bigcup_{B\in\mathcal{D}_{M_{n}}^*(\psi)}B$. Then
 \begin{align*}
 \lim E_n=\bigcap_{n=1}^\infty E_n=\lim\overline{E_n}=\bigcap_{n=1}^\infty\overline{E_n}\subset W(v)\cap K.
 \end{align*}
 \end{enumerate}
 \end{proposition}
\begin{proof}
The proof consists of two parts. In the first part, we construct $\{M_n\}_{n=0}^\infty$ and $\mathcal{D}_{M_n}^*(\psi)$ such that the conditions in \eqref{A} and \eqref{B} are satisfied. In the second part, we verify that the assertion in \eqref{C} holds. 

\medskip

\noindent\textbf{Part I. Construction of $\{M_n\}$ and $\mathcal{D}_{M_n}^*(\psi)$.}
Define $\tilde\psi:\mathbb{N}\rightarrow\mathbb{R}$ by 
$$\tilde\psi(q)=\frac{1}{3}\psi(q)=\frac{1}{3q^v}.$$
Let $\mathcal{I}^*, D_m(\psi), \mathcal{D}_m(\tilde\psi),$ and $D_m(\tilde\psi)$ be defined as in \S \ref{S 2}. Let $\kappa$ be the constant as in Theorem \ref{thm:BHZ}, and let $C,C'$, and $C_\mathcal{F}$ be the constants as in Proposition \ref{proposition 3.3}.

Since $\sum_{q=1}^\infty\psi(q)<\infty$ and $\sum_{q=1}^\infty\tilde\psi(q)<\infty$, by Lemma~\ref{lemma2}, there exists an integer $T_1$ such that the closures of the balls in $\mathcal{D}_m(\psi)$ and $\mathcal{D}_m(\tilde\psi)$ are disjoint, respectively, for any $m\ge T_1$. Since $v>1$ and $v-1<\kappa(v+1)/4<\kappa(v+1)/2$, $2^{-2m(v-1)}$ and $2^{m(-\kappa(v+1)/2)}$ decrease faster than $2^{-m(v-1)}$ as $m$ tends to $\infty$, so there exists an integer $T_2$ such that the following inequalities hold for any $m\ge T_2$:
 \begin{align}\label{C'}
 C'2^{-2m(v-1)}\le \frac{3C}{4}\cdot 2^{-m(v-1)},
 \end{align}
 and
 \begin{align}\label{M_1}
\left(\frac{c_1}{3\text{diam}(K)}\right)^{-\left(1+\frac{\kappa}{2}\right)}C_\mathcal{F}2^{m\left(-\frac{\kappa(v+1)}{2}\right)}\le\frac{C}{12}\cdot2^{-m(v-1)}.
\end{align}
Let $T_3$ be an integer such that $2^{-m(v+1)}\le\min \{3\text{diam}(K)/c_1, 1/12,c_1/4\}$ holds for any $m\ge T_3$. 

Let
$$u=\left\lceil \frac{4}{\kappa}+2\right\rceil+1,$$
where $\lceil\cdot\rceil$ is the ceiling function.
We set $M_0=\max\{T_1,T_2,T_3\}$ and $M_{n+1}=M_nu$ for all $n\in\mathbb{Z}_{\ge 0}$. It can be verified inductively that the following inequality holds for any $n\in\mathbb{N}$:
\begin{align}\label{firstneq}
\frac{C}{12}\cdot2^{-M_n(v-1)}\ge\left(\frac{c_1}{3\text{diam}(K)}\cdot 2^{-M_{n-1}(v+1)}\right)^{-\left(1+\frac{\kappa}{2}\right)} C_\mathcal{F} 2^{M_n\left(-\frac{\kappa(v+1)}{2}\right)}.
\end{align}
In the rest of this part of the proof, we use induction to construct the sequence of balls $\{\mathcal{D}_{M_n}^*(\psi)\}_{n=1}^\infty$.

\medskip

\noindent\textbf{Step 1. (Base of the Induction).}  
Let $L=\max\{1, 3\cdot2^{M_0(v+1)}\text{diam}(K)\}$. Let $x'$ be a point in $ K$, and let $r'=\max\{ 4\text{diam}(K), 2^{-M_0(v+1)}\}$, then $K$ is contained in $B(x',2r'/3)$. We set $M_{-1}=0$, $\mathcal{D}_{M_{-1}}^*(\psi)=\mathcal{D}_{M_0}^*(\psi)=\{B(x',r')\}$, and $\mathcal{I}_{-1}^*=\mathcal{I}_0^*=\{\emptyset\}$. 

With these choices, the following holds for $n=0$:
\begin{enumerate}[(i)]
\item\label{H2} $\mathcal{I}_n^*$ is a subset of $\mathcal{I}^*$.  
For every $\beta\in\mathcal{I}_n^*$, the branch $\mu_\beta$ of $\mu$ satisfies
$$\frac{2^{-M_n(v+1)}c_1}{3} \le\text{diam}(\text{supp}(\mu_\beta))\le \frac{2^{-M_n(v+1)}L}{3},$$
and there exists $\alpha\in\mathcal{I}_{n-1}^*$ such that $\alpha\prec\beta$.
\item\label{H3}  $\mathcal{D}_{M_n}^*(\psi)$ consists of 1-dimensional balls whose closures are disjoint, and any ball in $\mathcal{D}_{M_n}^*(\psi)$ has radius no less than $2^{-M_n(v+1)}$. 
\item\label{H4} 
There is a one-to-one correspondence between elements in $\mathcal{D}_{M_n}^*(\psi)$ and elements in $\mathcal{I}_n^*$: a ball $B(x,r)$ in $\mathcal{D}_{M_n}^*(\psi)$ corresponds to a unique word $\beta\in\mathcal{I}_n^*$ such that 
 $$\text{supp}(\mu_\beta)\subseteq B\Big(x,\frac{2r}{3}\Big).$$ 
 \item\label{H5} Every ball in $\mathcal{D}_{M_n}^*(\psi)$ is contained in some ball in $\mathcal{D}_{M_{n-1}}^*(\psi)$. Let $B'$ be a ball in $\mathcal{D}_{M_{n-1}}^*(\psi)$, and let $\alpha$ be the word in $\mathcal{I}_{n-1}^*$ that corresponds to $B'$ by the one-to-one correspondence in \eqref{H4}. Then, every ball in $\mathcal{D}_{M_n}^*(\psi)$ contained in $B'$ corresponds to a word in $\mathcal{I}_n^*$ that is greater than $\alpha$ with respect to the partial order on $\mathcal{I}^*$, and we have the following estimate on the number of balls in $\mathcal{D}_{M_n}^*(\psi)$ contained in $B'$:
 \begin{align*}
 \#\left\{B\in\mathcal{D}_{M_n}^*(\psi): B\subset B'\right\}\asymp2^{M_{n}(v+1)s-M_{n-1}(v+1)s-M_{n}(v-1)},
 \end{align*}
where the implicit constant depends only on the IFS $\mathcal{F}$.

\end{enumerate}

\medskip

\noindent\textbf{Step 2. (Main Induction Procedure).} Suppose that for all integers $0\le n\le k$, the set of words $\mathcal{I}^*_n$ and the set of balls $\mathcal{D}_{M_n}^*(\psi)$ have been determined, and conditions \eqref{H2}-\eqref{H5} are satisfied.

The next task is to determine $\mathcal{I}_{k+1}^*$ and $\mathcal{D}_{M_{k+1}}^*(\psi)$. 

Fix an element $\alpha\in\mathcal{I}_k^*$. By Proposition \ref{proposition 3.3}, we have
 \begin{align}\label{eqn7}
 \mu_\alpha(D_{M_{k+1}}(\psi))\le 12C\cdot 2^{-M_{k+1}(v-1)}+c_\alpha^{-\left(1+\frac{\kappa}{2}\right)}C_\mathcal{F}2^{M_{k+1}\left(-\frac{\kappa(v+1)}{2}\right)},
 \end{align}
 and
\begin{align}\label{eqn8}
\mu_\alpha(D_{M_{k+1}}(\tilde\psi))\ge\frac{C}{3}\cdot2^{-M_{k+1}(v-1)}-\frac{C'}{9}\cdot2^{-2M_{k+1}(v-1)}-3^{-\frac{\kappa}{2}}\cdot c_\alpha^{-\left(1+\frac{\kappa}{2}\right)} C_\mathcal{F} 2^{M_{k+1}\left(-\frac{\kappa(v+1)}{2}\right)}.
\end{align}
 
 Recall that $\text{supp}(\mu_\alpha)=f_\alpha(K)$ and $f_\alpha(x)=c_\alpha x+b_\alpha$, so $\text{diam}(\text{supp}(\mu_\alpha))=c_\alpha\text{diam}(K)$. Then by the induction hypothesis \eqref{H2}, we have
\begin{align}\label{cL}
\frac{c_1}{3\text{diam}(K)}\cdot 2^{-M_k(v+1)}\le c_\alpha\le \frac{L}{3\text{diam}(K)}\cdot 2^{-M_k(v+1)}.
\end{align}
This lower bound for $c_\alpha$ together with \eqref{firstneq} for $n=k+1$ implies
\begin{align}
\label{main}\frac{C}{12}\cdot2^{-M_{k+1}(v-1)}\ge c_\alpha^{-\left(1+\frac{\kappa}{2}\right)}C_\mathcal{F}2^{M_{k+1}\left(-\frac{\kappa(v+1)}{2}\right)}.
\end{align}
Since $M_{k+1}> M_0\ge \max\{T_1,T_2\}$, the closures of balls in $\mathcal{D}_{M_{k+1}}(\psi)$ and $\mathcal{D}_{M_{k+1}}(\tilde\psi)$ are disjoint, respectively, and \eqref{C'} holds for $m=M_{k+1}$, i.e., we have
\begin{align}\label{main2}
C'2^{-2M_{k+1}(v-1)}\le\frac{3C}{4} \cdot2^{-M_{k+1}(v-1)}.
\end{align}
By \eqref{main} and \eqref{main2}, in the estimates of $\mu_\alpha(D_{M_{k+1}}(\psi))$ and $\mu_\alpha(D_{M_{k+1}}(\tilde\psi))$, the main terms of the form $constant\cdot2^{-M_{k+1}(v-1)}$ majorize the rest error terms, and we have
 %\frac{C}{12}\cdot2^{-M_{k+1}(v-1)}\le
$$\mu_\alpha(D_{M_{k+1}}(\psi))\le 13C\cdot 2^{-M_{k+1}(v-1)},$$ and
$$\mu_\alpha(D_{M_{k+1}}(\tilde\psi))\ge \frac{C}{6}\cdot2^{-M_{k+1}(v-1)}.$$

Let $\mathcal{D}_{M_{k+1},\alpha}^*(\tilde\psi)$ be the collection of balls in $\mathcal{D}_{M_{k+1}}(\tilde\psi)$ that have nonempty intersection with $\text{supp}(\mu_\alpha)$. By the disjointness of balls in $\mathcal{D}_{M_{k+1}}(\psi)$, a ball in $\mathcal{D}_{M_{k+1}}(\tilde\psi)$ is contained in the unique ball in $\mathcal{D}_{M_{k+1}}(\psi)$ centered at the same point with a tripled radius. Let $\mathcal{D}_{M_{k+1},\alpha}^*(\psi)$ be the collection of balls in $\mathcal{D}_{M_{k+1}}(\psi)$ that contain balls in $\mathcal{D}_{M_{k+1},\alpha}^*(\tilde\psi)$, and let
\begin{align*}
D_{M_{k+1},\alpha}^*(\psi)=\bigcup_{B\in\mathcal{D}_{M_{k+1},\alpha}^*(\psi)}B\text{ and }D_{M_{k+1},\alpha}^*(\tilde\psi)=\bigcup_{B\in\mathcal{D}_{M_{k+1},\alpha}^*(\tilde\psi)}B.
\end{align*}
Since $D_{M_{k+1}}(\tilde\psi)\setminus D_{M_{k+1},\alpha}^*(\tilde\psi)$ is the union of balls in $\mathcal{D}_{M_{k+1}}(\tilde\psi)$ that have no intersection with $\text{supp}(\mu_\alpha)$, its $\mu_\alpha$-measure is zero, thus
\begin{align}\label{mu_alpha(tilde)}
\mu_\alpha(D_{M_{k+1},\alpha}^*(\tilde\psi))=\mu_\alpha(D_{M_{k+1}}(\tilde\psi))\ge \frac{C}{6}\cdot2^{-M_{k+1}(v-1)} .
\end{align}
Since $D_{M_{k+1},\alpha}^*(\psi)$ is a subset of $D_{M_{k+1}}(\psi)$, we have
\begin{align}\label{mu_alpha}
\mu_\alpha(D_{M_{k+1},\alpha}^*(\psi))\le\mu_\alpha(D_{M_{k+1}}(\psi))\le  13C\cdot 2^{-M_{k+1}(v-1)}.
\end{align}

Let $B(x,r)$ be a ball in $\mathcal{D}_{M_{k+1},\alpha}^*(\psi)$. Then, $B(x,r/3)$ is a ball in $\mathcal{D}_{M_{k+1},\alpha}^*(\tilde\psi)$ and has nonempty intersection with $\text{supp}(\mu_\alpha)$. Recall that by definition, $x$ and $r$ are of the form
$$x=\frac{p}{q} \text{ and } r=\frac{\psi(2^{M_{k+1}})}{q},\text{ with }(p,q)\in P(\mathbb{Z}^2)\text{ and }  2^{M_{k+1}-1}\le q<2^{M_{k+1}},$$
so $r$ is between $2^{-M_{k+1}(v+1)}$ and $2\cdot2^{-M_{k+1}(v+1)}$. By the induction hypothesis (iv), there exists a ball, which we denote by $B(x_0,r_0)$, in $\mathcal{D}_{M_k}^*(\psi)$ such that $\text{supp}(\mu_\alpha)\subset B(x_0,2r_0/3)$. By the induction hypothesis (iii), $r_0\ge 2^{-M_k(v+1)}$. Since $B(x,r)$ has nonempty intersection with $\text{supp}(\mu_\alpha)$, $B(x,r)$ is contained in $B(x_0,2r_0/3+2r)$. Note that the common ratio of $\{M_k\}$ is at least 2, so we have
$$M_{k+1}-M_k\ge M_k\ge M_0,$$
hence
$$\frac{1}{2^{(M_{k+1}-M_k)(v+1)}}\le\frac{1}{2^{M_0(v+1)}}<\frac{1}{12},$$
where the last inequality holds by our choice of $M_0$. Then, we have
$$2r\le2\cdot\frac{2}{2^{M_{k+1}(v+1)}}<\frac{1}{3\cdot 2^{M_k(v+1)}}\le \frac{r_0}{3}, $$
which then implies $B(x,r)$ is contained in $B(x_0,r_0)$. Since this result holds for every ball in $\mathcal{D}_{M_{k+1},\alpha}^*(\psi)$, we then conclude that $D_{M_{k+1},\alpha}^*(\psi)$ is a subset of $B(x_0,r_0)$.

Take a point $y\in B(x,r/3)\cap \text{supp}(\mu_\alpha)$. By the induction hypothesis \eqref{H2}, we have
$$\frac{\text{diam}(\text{supp}(\mu_\alpha))}{c_1}\ge \frac{2^{-M_k(v+1)}}{3}\ge\frac{2^{-M_{k+1}(v+1)}}{3},$$
so the condition in Lemma~\ref{lemma 1.4} is satisfied with $Q=2^{-M_{k+1}(v+1)}/3$. By Lemma~\ref{lemma 1.4}, there exists a branch $\mu_\beta$ of $\mu$ such that $\alpha\prec\beta$, $y\in\text{supp}(\mu_\beta)$, and 
\begin{align}\label{condition1}
\text{diam}(\text{supp}(\mu_\beta))\in[c_12^{-M_{k+1}(v+1)}/3,2^{-M_{k+1}(v+1)}/3].
\end{align}
The requirement on the diameter of $\text{supp}(\mu_\beta)$ ensures that
\begin{align}\label{condition2}
\text{supp}(\mu_\beta)\subset B\Big(x,\frac{2r}{3}\Big).
\end{align}
Since $y\in B(x,r/3)$, we have
\begin{align}\label{inclusion of B_y}
B\left(x,\frac{r}{3}\right)\subseteq B\left(y,\frac{2r}{3}\right) \subset B(x,r),
\end{align}
thus
\begin{align*}
\mu_\alpha\Big( B\Big(x,\frac{r}{3}\Big)\Big)\le\mu_\alpha\Big(B\Big(y,\frac{2r}{3}\Big)\Big)\le\mu_\alpha(B(x,r)).
\end{align*}
Since $u\ge2$ and $M_k\ge M_0\ge T_3$, $4\cdot 2^{-(M_{k+1}-M_k)(v+1)}=4\cdot 2^{-M_k(u-1)(v+1)}\le c_1$, hence
$$\frac{2r}{3}\le\frac{4\cdot 2^{-M_{k+1}(v+1)}}{3}\le\frac{2^{-M_k(v+1)}c_1}{3}\le \mathrm{diam}(\mathrm{supp}(\mu_\alpha)).$$
By Lemma \ref{regularity lemma}, we have
\begin{align}\label{r_alpha}
a_1c_\alpha^{-s}\left(\frac{2r}{3}\right)^s\le\mu_\alpha\Big(B\Big(y,\frac{2r}{3}\Big)\Big)\le a_2c_\alpha^{-s}\left(\frac{2r}{3}\right)^s.
\end{align}
This estimate, together with \eqref{cL} and the ranges of $r$, gives us
\begin{align*}
\mu_\alpha\Big(B\Big(y,\frac{2r}{3}\Big)\Big)\asymp 2^{M_k(v+1)s-M_{k+1}(v+1)s},
\end{align*}
where the implicit constant depends only on $\mathcal{F}$. It follows that
\begin{align}\label{r}
\mu_\alpha(B(x,r))\gtrsim2^{M_k(v+1)s-M_{k+1}(v+1)s}.
\end{align}
and
\begin{align}\label{r/3}
\mu_\alpha\Big( B\Big(x,\frac{r}{3}\Big)\Big)\lesssim2^{M_k(v+1)s-M_{k+1}(v+1)s}
\end{align}
By varying $B(x,r)$ over all the elements in $\mathcal{D}_{M_{k+1},\alpha}^*(\psi)$, we pair every ball $B(x,r)$ in $\mathcal{D}_{M_{k+1},\alpha}^*(\psi)$ with a branch $\mu_{\beta}$ of $\mu$ such that $\alpha\prec\beta$ and conditions~\eqref{condition1} and \eqref{condition2} are satisfied. We collect these branches of $\mu$ and let $\mathcal{I}_{k+1,\alpha}^*$ be the collection of the indices of these branches.

Now we use the $\mu_\alpha$-measure estimates of $D_{M_{k+1}}(\psi)$ and $D_{M_{k+1}}(\tilde\psi)$ to count the number of balls in $\mathcal{D}^*_{M_{k+1},\alpha}(\psi)$.
Since the bounds in the estimates \eqref{r} and \eqref{r/3} are independent of the ball $B(x,r)$ we choose,
the $\mu_\alpha$-measure of any ball in $\mathcal{D}_{M_{k+1},\alpha}^*(\psi)$ is $\gtrsim2^{M_k(v+1)s-M_{k+1}(v+1)s}$ and the $\mu_\alpha$-measure of any ball in $\mathcal{D}_{M_{k+1},\alpha}^*(\tilde\psi)$ is $\lesssim 2^{M_k(v+1)s-M_{k+1}(v+1)s}$. With the estimates \eqref{mu_alpha(tilde)} and \eqref{mu_alpha}, we compute the cardinalities of $\mathcal{D}_{M_{k+1},\alpha}^*(\psi)$ and $\mathcal{D}_{M_{k+1},\alpha}^*(\tilde\psi)$:
\begin{align*}
\#\mathcal{D}_{M_{k+1},\alpha}^*(\psi)&\lesssim\frac{\mu_\alpha\left(D^*_{M_{k+1,\alpha}}(\psi)\right)}{2^{M_k(v+1)s-M_{k+1}(v+1)s}}\\
&\lesssim 2^{M_{k+1}(v+1)s-M_k(v+1)s-M_{k+1}(v-1)},
\end{align*}
and
\begin{align*}
\#\mathcal{D}_{M_{k+1},\alpha}^*(\tilde\psi)&\gtrsim\frac{\mu_\alpha\left(D^*_{M_{k+1,\alpha}}(\tilde\psi)\right)}{2^{M_k(v+1)s-M_{k+1}(v+1)s}}\\
&\gtrsim  2^{M_{k+1}(v+1)s-M_k(v+1)s-M_{k+1}(v-1)}.
\end{align*}
Since the cardinalities of $\mathcal{D}_{M_{k+1},\alpha}^*(\tilde\psi)$ and $\mathcal{D}_{M_{k+1},\alpha}^*(\psi)$ are the same, we then conclude that
\begin{align}\label{cardinality}
\#\mathcal{D}_{M_{k+1},\alpha}^*(\psi)\asymp2^{M_{k+1}(v+1)s-M_k(v+1)s-M_{k+1}(v-1)},
\end{align}
where the implicit constant depends only on $\mathcal{F}$. Repeating the same process for every $\alpha$ in $\mathcal{I}_{k}^*$, we obtain the corresponding $\mathcal{I}^*_{k+1,\alpha}$ and $\mathcal{D}_{M_{k+1},\alpha}^*(\psi)$ for every $\alpha\in\mathcal{I}_k^*$.

We set
$$\mathcal{I}_{k+1}^*=\bigcup_{\alpha\in\mathcal{I}_k^*}\mathcal{I}_{k+1,\alpha}^*.$$
Since $L>1$, the upper bound in \eqref{condition1} can be replaced by $2^{-M_{k+1}(v+1)}L/3$, so we have
\begin{align*}
\text{diam}(\text{supp}(\mu_\beta))\in[c_12^{-M_{k+1}(v+1)}/3,2^{-M_{k+1}(v+1)}L/3].
\end{align*}
By the construction of $\mathcal{I}_{k+1,\alpha}^*$, every element in $\mathcal{I}^*_{k+1,\alpha}$ is greater than $\alpha$ with respect to the partial order on $\mathcal{I}^*$. Hence, the condition in the induction hypothesis \eqref{H2} is satisfied with the set $\mathcal{I}_{k+1}^*$ constructed as above. 
We set
$$\mathcal{D}_{M_{k+1}}^*(\psi)=\bigcup_{\alpha\in\mathcal{I}_{k}^*}\mathcal{D}_{M_{k+1},\alpha}^*(\psi).
$$
Then, $\mathcal{D}_{M_{k+1}}^*(\psi)$ is a subset of $\mathcal{D}_{M_{k+1}}(\psi)$, which consists of disjoint balls of radii at least $2^{-M_{k+1}(v+1)}$, so the condition in the induction hypothesis \eqref{H3} is satisfied with $\mathcal{D}_{M_{k+1}}^*(\psi)$ constructed as above. By \eqref{condition2}, %and our choices of $\mathcal{I}_{k+1}^*$ and $\mathcal{D}_{M_{k+1}}^*(\psi)$, 
the condition in the induction hypothesis \eqref{H4} is satisfied. For every $\alpha\in\mathcal{I}_k^*$, the set $\mathcal{D}_{k+1,\alpha}^*(\psi)$ is contained in the unique ball in $\mathcal{D}_{k}^*(\psi)$ that corresponds to $\alpha$ under one-to-one correspondence in \eqref{H4} and the cardinality of $\mathcal{D}_{M_{k+1},\alpha}^*(\psi)$ satisfies \eqref{cardinality}. Moreover, the words corresponding to balls in $\mathcal{D}_{k+1,\alpha}^*(\psi)$ constitute $\mathcal{I}_{k+1,\alpha}^*$ and are all greater than $\alpha$.  
Hence, the condition in the induction hypothesis \eqref{H5} is satisfied. This finishes the induction. 

In the construction, for $n\ge 1$, $\mathcal{D}_{M_n}^*(\psi)$ is a subset of $\mathcal{D}_{M_n}(\psi)$, and any ball $B(x,r)$ in $\mathcal{D}_{M_{n}}^*(\psi)$ contains a ball $B(y,2r/3)$ with $y\in\text{supp}(\mu)$. The regularity property of $\mu$ then implies that
\begin{align*}
\mu(B(x,r))\ge\mu\left(B\left(y,\frac{2r}{3}\right)\right)\ge \frac{2^sa_1}{3^s}r^s.
\end{align*}
The above inequality, together with the induction hypothesis \eqref{H3}, verifies part \eqref{A} of the proposition. The induction hypothesis \eqref{H5} verifies part \eqref{B} of the proposition.

\medskip

\noindent\textbf{Part II. Inclusion of the limit sets in $W(v)\cap K$.}
First, we show that $\lim E_n$ is contained in $W(v)\cap K$.
By part \eqref{B}, $\{E_n\}$ is a decreasing sequence, so 
\begin{align*}
\lim E_n=\limsup E_n=\liminf E_n=\bigcap_{n=1}^\infty E_n.
\end{align*}
Note that $E_n$ is contained in $D_{M_n}(\psi)$, so $\lim E_n\subset \limsup D_m(\psi)$. To show $\lim E_k\subset W(v)$, it suffices to show that $\limsup D_m(\psi)\subset W(v)$. Suppose $x\in\limsup D_m(\psi)$, then there are infinitely many $m\in\mathbb{N}$ and $(p,q)\in P(\mathbb{Z}^{2})$ such that $2^{m-1}\le q<2^m$ and $|qx-p|<\psi(2^m)\le\psi(q)$, so $x\in W(v)$ by the definition of $W(v)$, thus $\limsup D_m(\psi)\subset W(v)$.

For every $n\ge 1$, we enumerate the elements in $\mathcal{I}_n^*$ and write 
$$\mathcal{I}^*_n=\{\alpha_j^n:1\le j\le \#\mathcal{I}_n^*\}.$$ 
For $\alpha_j^n\in\mathcal{I}_n^*$, denote by $B_{\alpha_j^n}$ the ball in $\mathcal{D}_{M_n}^*(\psi)$ that the corresponds to $\alpha_j^n$ by the one-to-one correspondence in the induction hypothesis \eqref{H4}.
Then, $B_{\alpha_j^n}$ contains $\mathrm{supp}(\mu_{\alpha_j^n})$ and we have
\begin{align}
\notag\lim E_n=\bigcap_{n=1}^\infty E_n&=\bigcap_{n=1}^\infty \bigcup_{\alpha^n_j\in\mathcal{I}^*_n}B_{\alpha^n_j}\\
&=\bigcup_{(\alpha_{j_n}^n)\in\prod_{n=1}^\infty\mathcal{I}^*_n}\,\bigcap_{n=1}^\infty B_{\alpha_{j_n}^n}\label{E_n},
\end{align}
where $(\alpha_{j_n}^n)$ denotes the element in $\prod_{n=1}^\infty\mathcal{I}^*_n$ whose $n$-th coordinate is $\alpha_{j_n}^n$. 
By the induction hypothesis \eqref{H5}, $B_{\alpha_{j_n}^n}\cap B_{\alpha_{j_{n+1}}^{n+1}}$ is nonempty if and only if $\alpha_{j_{n}}^{n}\prec\alpha_{j_{n+1}}^{n+1}$. Therefore, a necessary condition for $\bigcap_{n=1}^\infty B_{\alpha_{j_n}^n}$ to be nonempty is that $(\alpha_{j_n}^n)\in\prod_{n=1}^\infty\mathcal{I}^*_n$ satisfies 
\begin{align}\label{totally ordered}
\alpha^1_{j_1}\prec\alpha^2_{j_2}\prec\cdots\prec\alpha^n_{j_n}\prec\alpha^{n+1}_{j_{n+1}}\prec\cdots.
\end{align}
Let
$$\left(\prod_{n=1}^\infty\mathcal{I}_n^*\right)^\times=\left\{(\alpha^n_{j_n})\in\prod_{n=1}^\infty\mathcal{I}^*_n: (\alpha^n_{j_n}) \text{ satisfies }~\eqref{totally ordered}\right\}.$$
Then, ~\eqref{E_n} can be reduced to
\begin{align}\label{union}
\lim E_n=\bigcup_{(\alpha_{j_n}^n)\in\left(\prod_{n=1}^\infty\mathcal{I}^*_n\right)^\times}\,\bigcap_{n=1}^\infty B_{\alpha_{j_n}^n}.
\end{align}
Suppose $(\alpha_{j_n}^n)$ is an element in $(\prod_{n=1}^\infty\mathcal{I}^*_n)^\times$, then $\mu_{\alpha_{j_{n+1}}^{n+1}}$ is a branch of $\mu_{\alpha_{j_{n}}^{n}}$ for every $n\ge 1$, so $\text{supp}(\mu_{\alpha_{j_{n+1}}^{n+1}})\subset\text{supp}(\mu_{\alpha_{j_{n}}^{n}})$. Therefore, $\{\text{supp}(\mu_{\alpha_{j_n}^n})\}$ is a decreasing sequence of closed sets, which then satisfies the finite intersection property. Since $\text{supp}(\mu_{\alpha_{j_n}^n})\subset\text{supp}(\mu)$ for any $n\ge 1$ and $\text{supp}(\mu)$ is compact, we have
$$\bigcap_{n=1}^\infty B_{\alpha_{j_n}^n}\supset\bigcap_{n=1}^\infty \text{supp}(\mu_{\alpha_{j_n}^n})=\text{nonempty subset of }\text{supp}(\mu).$$
As $n$ goes to $\infty$, the diameter of $B_{\alpha_{j_n}^n}$ tends to zero, so $\bigcap_{n=1}^\infty B_{\alpha_{j_n}^n}$ contains at most one point. It follows that 
$$\bigcap_{n=1}^\infty B_{\alpha_{j_n}^n}=\bigcap_{n=1}^\infty \text{supp}(\mu_{\alpha_{j_n}^n})= \text{a point in }\text{supp}(\mu).$$
Then by \eqref{union}, $\lim E_n$ is a union of points in $\text{supp}(\mu)=K$. We then conclude that $\lim E_n\subset W(v)\cap K$.

Next we prove $\lim E_n=\lim \overline{E_n}$. Since $\{E_n\}$ is a sequence of decreasing sets, $\{\overline{E_n}\}$ is also a sequence of decreasing sets, so we have
$$\lim \overline{E_n}=\bigcap_{n=1}^\infty \overline{E_n}.$$
For every $n\ge 1$, $E_n=\bigcup_{\alpha_j^n\in\mathcal{I}_n^*}B_{\alpha_j^n}$ is a union of finitely many balls whose closures are disjoint, so 
$$\overline{E_n}=\overline{\bigcup_{\alpha_j^n\in\mathcal{I}_n^*}B_{\alpha_j^n}}=\bigcup_{\alpha_j^n\in\mathcal{I}_n^*}\overline{B_{\alpha_j^n}}.$$
Thus, we have
\begin{align*}
\lim \overline{E_n}=\bigcap_{n=1}^\infty \overline{E_n}&=\bigcap_{n=1}^\infty\bigcup_{\alpha_j^n\in\mathcal{I}_n^*}\overline{B_{\alpha_j^n}}\\
&=\bigcup_{(\alpha_{j_n}^n)\in\prod_{n=1}^\infty\mathcal{I}^*_n}\,\bigcap_{n=1}^\infty \overline{B_{\alpha_{j_n}^n}}.
\end{align*}
By the previous proof, $B_{\alpha_{j_n}^n}\cap B_{\alpha_{j_{n+1}}^{n+1}}\ne\emptyset$ if and only if $\alpha_{j_{n}}^{n}\prec\alpha_{j_{n+1}}^{n+1}$. 
Since the closures of balls in $\mathcal{D}_{M_{n}}^*(\psi)$ are disjoint and every ball in $\mathcal{D}_{M_{n+1}}^*(\psi)$ is contained in a unique ball in $\mathcal{D}_{M_n}^*(\psi)$, $\overline{B_{\alpha_{j_n}^n}}\cap \overline{B_{\alpha_{j_{n+1}}^{n+1}}}\ne\emptyset$ if and only if $\alpha_{j_{n}}^{n}\prec\alpha_{j_{n+1}}^{n+1}$. Thus, a necessary condition for $\bigcap_{n=1}^\infty \overline{B_{\alpha_{j_n}^n}}$ to be nonempty is 
\begin{align*}
\alpha^1_{j_1}\prec\alpha^2_{j_2}\prec\cdots\prec\alpha^n_{j_n}\prec\alpha^{n+1}_{j_{n+1}}\prec\cdots,
\end{align*}
which is the same as the necessary condition for $\bigcap_{n=1}^\infty B_{\alpha_{j_n}^n}$ to be nonempty. So, we have 
\begin{align}\label{closure}
\lim \overline{E_n}=\bigcup_{\left(\alpha_{j_n}^n)\in(\prod_{n=1}^\infty\mathcal{I}^*_n\right)^\times}\,\bigcap_{n=1}^\infty \overline{B_{\alpha_{j_n}^n}}.
\end{align}
Let $(\alpha_{j_n}^n)$ be an element in $\left(\prod_{n=1}^\infty\mathcal{I}^*_n\right)^\times$. Then $\bigcap_{n=1}^\infty \overline{B_{\alpha_{j_n}^n}}$ is nonempty since it contains the nonempty set $\bigcap_{n=1}^\infty B_{\alpha_{j_n}^n}$. As $n$ goes to $\infty$, the diameter of $\overline{B_{\alpha_{j_n}^n}}$ tends to zero, so $\bigcap_{n=1}^\infty \overline{B_{\alpha_{j_n}^n}}$ contains at most one point and this forces
$$\bigcap_{n=1}^\infty \overline{B_{\alpha_{j_n}^n}}=\bigcap_{n=1}^\infty B_{\alpha_{j_n}^n}.$$
The above identity together with \eqref{union} and \eqref{closure} implies that $\lim E_n=\lim\overline{E_n}$. This proves part \eqref{C} of the proposition.
\end{proof}

Next, we construct a measure whose support is contained in $W(v)\cap K$, and then apply the Mass Distribution Principle, which we recall below, to find a lower bound for the Hausdorff dimension of $W(v)\cap K$. 

\begin{lemma}[{Mass Distribution Principle \cite[Theorem 4.2]{Fal}}] \label{MDP}
Let $\lambda$ be a Borel probability measure supported on a set $F\subset\mathbb{R}^n$. Suppose for some $t$ there are numbers $c>0$ and $\epsilon>0$ such that 
\begin{align}\label{eq:Frostman}
    \lambda(U)\le c\cdot\textup{diam}(U)^t,
\end{align}
for all Borel sets $U$ with $\textup{diam}(U)\le\epsilon$. Then $\dim_HF\ge t$.
\end{lemma}

\begin{proposition}\label{bound for mU}
Suppose $\psi(q)=1/q^v$ satisfies the assumption in Proposition \ref{main construction}. Let $\{M_n\}$ and $\{E_n\}$ be as in Proposition \ref{main construction}, with $u$ being the common ratio of $\{M_n\}$. Then, there exists a Borel probability measure $m$ that is supported on $\lim E_n$ and satisfies the following property: for every $t<s-u(v-1)/(v+1)$, there exists $\epsilon>0$, depending on $t$, such that
\begin{align}\label{Frostman}
m(U)\le \textup{diam}(U)^t
\end{align}
for all Borel sets $U$ with $\textup{diam}(U)\le \epsilon$. 
\end{proposition}
\begin{proof}
We divide the proof into two parts. In the first part, we construct the desired Borel probability measure $m$ whose support is contained in $\lim E_n$, and in the second part, we verify that \eqref{Frostman} is satisfied.
\medskip

\noindent\textbf{Part I. Construction of the measure $m$.} Let
 $$\mathcal{E}_0=\{\mathbb{R}\},\quad\mathcal{E}_n=\mathcal{D}_{M_n}^*(\psi)\cup\{\mathbb{R}\setminus E_n\}\, \,\text{for} \,\,n\in\mathbb{N},\quad\text{and } \mathcal{E}=\bigcup_{n=0}^\infty \mathcal{E}_n,$$
 where the sets $\mathcal{D}_{M_n}^*(\psi)$'s are as in Proposition \ref{main construction}. We view the intervals in $\mathcal{E}_n$ as 1-dimensional balls. To construct the measure $m$, we first define a function $\tilde{m}:\mathcal{E}\rightarrow [0,1]$ inductively. For $k=0$, define $\tilde m(\mathbb{R})=1$. Suppose, for every $0\le n\le k$, $\tilde m:\mathcal{E}_n\rightarrow[ 0,1]$ is determined. To define $\tilde m:\mathcal{E}_{k+1}\rightarrow [0,1]$, we split the value $\tilde{m}(B)$ of every ball $B$ in $\mathcal{E}_k$ equally between balls in $\mathcal{E}_{k+1}$ that are contained in $B$ and set $\tilde{m}(\mathbb{R}\setminus E_{k+1})=0$. For $k=0$, every ball in $\mathcal{E}_1$ is contained in $\mathbb{R}$, and for $k\ge 1$,  every ball in $\mathcal{E}_{k+1}$ is an element in $\mathcal{D}_{M_{k+1}}^*(\psi)$ and thus is contained in some ball in $\mathcal{D}_{M_k}^*(\psi)\subset\mathcal{E}_k$ by part \eqref{B} of Proposition \ref{main construction}, so $\tilde{m}:\mathcal{E}_{k+1}\rightarrow[0,1]$ is defined. Repeating this process, we obtain the desired function $\tilde{m}:\mathcal{E}\rightarrow[0,1]$. 

By part \eqref{B} of Proposition \ref{main construction}, there is an estimate of the number of balls in $\mathcal{D}_{M_{n+1}}^*(\psi)$ contained in a single ball in $\mathcal{D}_{M_n}^*(\psi)$, uniformly in $n$. Moreover, since $\mathcal{D}_{M_1}^*(\psi)$ is finite, we could choose a uniform constant $\eta>0$ such that
$$\#\mathcal{D}_{M_1}^*(\psi)\ge\eta 2^{M_1(v+1)s-M_0(v+1)s-M_1(v-1)}$$
holds and that
$$ \#\{B\in\mathcal{D}_{M_{n+1}}^*(\psi):B\subset B'\}\ge \eta2^{M_{n+1}(v+1)s-M_{n}(v+1)s-M_{n+1}(v-1)}$$
holds for any $n\ge 1$ and any $B'\in\mathcal{D}_{M_n}^*(\psi)$.  
Let $B^{(n+1)}$ be an arbitrary ball in $\mathcal{D}_{M_{n+1}}^*(\psi)$, and let $B^{(n)}$ be the ball in $\mathcal{D}_{M_n}^*(\psi)$ that contains $B^{(n+1)}$. By our definition of $\tilde{m}$, we have
$$\tilde m\left(B^{(n+1)}\right)\le \eta^{-1}2^{-M_{n+1}(v+1)s+M_{n}(v+1)s+M_{n+1}(v-1)} \tilde m\left(B^{(n)}\right).$$
Iterating the above process backwardly, we obtain
\begin{align}
\notag\tilde m\left(B^{(n+1)}\right)&\le \eta^{-n-1}\prod_{j=1}^{n+1}2^{-M_{j}(v+1)s+M_{j-1}(v+1)s+M_{j}(v-1)}\tilde m(\mathbb{R})\\
&\label{measure}=\eta^{-n-1}\prod_{j=1}^{n+1}2^{-M_{j}(v+1)s+M_{j-1}(v+1)s+M_{j}(v-1)}.
\end{align}
For any Borel subset $A$ of $\mathbb{R}$, define 
$$m(A)=\inf\left\{\sum_{i=1}^\infty\tilde m(U_i):A\subset\bigcup_{i=1}^\infty U_i\text{ and } U_i\in\mathcal{E}\text{ for all }i\ge 1\right\}.$$By \cite[Proposition 1.7]{Fal}, $m$ is a Borel probability measure on $\mathbb{R}$ which agrees with $\tilde m$ on $\mathcal{E}$, and $\text{supp}(m)$ is contained in $\bigcap_{n=1}^\infty \overline{E_n}$. Then by part \eqref{C} of Proposition \ref{main construction}, $\text{supp}(m)$ is contained in $\lim E_n.$
\medskip

\noindent\textbf{Part II. Verifying Frostman's condition \eqref{Frostman} for $m$.}
Let $U$ be a Borel subset of $\mathbb{R}$ with $\text{diam}(U)=R$. For now, we assume that $R$ is a small number strictly less than $2^{-M_{1}(v+1)}$. Then, there is an integer $p\ge 2$ such that $2^{-M_p(v+1)}\le R<2^{-M_{p-1}(v+1)}$. Without loss of generality, we may assume $U\cap\lim E_n\ne\emptyset$, otherwise $m(U)=0$ and so $m(U)\le R^t$ holds trivially for any $t>0$. Let $z$ be a point in $U\cap\lim E_n$. Then, $U$ is contained in $B(z,7R)$. By the regularity property of $\mu$, we have
\begin{align}\label{upperBz}
\mu(B(z,7R))\le 7^sa_2R^s.
\end{align}
Let $B(x,r)$ be a ball in $\mathcal{D}_{M_p}^*(\psi)$ that has non-empty intersection with $U$. Since $B(x,r)$ is an element of $\mathcal{D}_{M_p}(\psi)$, $2^{-M_p(v+1)}<r\le 2\cdot 2^{-M_p(v+1)}$. By part \eqref{A} of Proposition \ref{main construction}, there is a lower bound for the $\mu$-measure of $B(x,r)$:
\begin{align*}
\mu\left(B(x,r)\right)\ge\frac{2^sa_1 }{3^s}r^s&\ge \frac{2^s a_1}{3^s}\cdot 2^{-M_p(v+1)s}.
\end{align*}
Since the diameter of $B(x,r)$ is $2r\le 4R$, $B(x,r)$ is contained $B(z,7R)$.
Let $\mathcal{B}=\{B_1,B_2,\cdots,B_l\}$ be the collection of balls in $\mathcal{D}_{M_p}^*(\psi)$ that have non-empty intersection with $U$.
Since the balls in $\mathcal{D}_{M_p}^*(\psi)$ are disjoint, we have
\begin{align*}
\mu\left(B(z,7R)\right)\ge \sum_{i=1}^l\mu(B_i)\ge l\cdot \frac{2^s a_1}{3^s}\cdot 2^{-M_p(v+1)s}.
\end{align*}
Using the upper bound for $\mu(B(z,7R))$ in \eqref{upperBz}, we obtain
\begin{align}\label{l}
l\le\frac{\mu\left(B(z,7R)\right)}{\frac{2^s a_1}{3^s}\cdot 2^{-M_p(v+1)s}}\le\frac{7^sa_2R^s}{\frac{2^s a_1}{3^s}\cdot 2^{-M_p(v+1)s}}=\frac{21^sa_2}{2^sa_1}\cdot 2^{M_p(v+1)s}R^s.
\end{align}
The subset $\mathcal{B}\cup\{\mathbb{R}\setminus E_p\}$ of $\mathcal{E}_p$ forms a cover of $U$, so by \eqref{measure}, \eqref{l}, and the definition of $m$, we have
\begin{align*}
m(U)&\le \sum_{i=1}^l\tilde m(B_i)+\tilde{m}(\mathbb{R}\setminus E_p)\\
&\le\frac{21^sa_2}{2^sa_1}\cdot 2^{M_p(v+1)s}R^s\eta^{-p}\prod_{j=1}^{p}2^{-M_{j}(v+1)s+M_{j-1}(v+1)s+M_{j}(v-1)}.
\end{align*}
To obtain the inequality $m(U)\le R^t$, we consider the quotient $\log_2 m(U)/\log_2 R$.
Note that $R<2^{-M_1(v+1)}$ and so $\log_2R<0$. We have
\begin{align}
\notag\frac{\log_2 m(U)}{\log_2 R}&\ge \frac{\log_2\left(\frac{21^sa_2}{2^sa_1}\cdot 2^{M_p(v+1)s}R^s\eta^{-p}\prod_{j=1}^{p}2^{-M_{j}(v+1)s+M_{j-1}(v+1)s+M_{j}(v-1)}\right)}{\log_2 R}\\
&\label{ineq13}=\frac{\log_2 \left(\frac{21^sa_2}{2^sa_1\eta^p}\right)}{\log_2 R}+s+\frac{M_p(v+1)s+\sum_{j=1}^p\left(-M_{j}(v+1)s+M_{j-1}(v+1)s+M_{j}(v-1)\right)}{\log_2 R}
\end{align}

We analyze the terms in \eqref{ineq13} separately. 
Recall that $\{M_{n}\}$ is a geometric series of common ration $u$, so $M_n=M_0u^n,$ then
\begin{align*}
&M_p(v+1)s+\sum_{j=1}^p(-M_{j}(v+1)s+M_{j-1}(v+1)s+M_{j}(v-1))\\
=&\sum_{j=1}^p M_j(v-1)+M_0(v+1)s=\frac{M_0u(u^p-1)}{u-1}(v-1)+M_0(v+1)s.
\end{align*}
Since $R\le 2^{-M_{p-1}(v+1)}$, $\log_2 R\le -M_{p-1}(v+1)=-M_0u^{p-1}(v+1)$. The third term in \eqref{ineq13} is bounded below by 
\begin{align*}
&\frac{\frac{M_0u(u^p-1)}{u-1}(v-1)+M_0(v+1)s}{-M_0u^{p-1}(v+1)}=-\frac{u^p-1}{u^{p-2}(u-1)}\cdot\frac{v-1}{v+1}-\frac{s}{u^{p-1}}.
\end{align*}
For the first term in \eqref{ineq13}, we have
\begin{align*}
\left|\frac{\log_2 \left(\frac{21^sa_2}{2^sa_1\eta^p}\right)}{\log_2 R}\right|
&\le
\frac{\log_2\left(\frac{21^sa_2}{2^sa_1}\right)+p\left|\log_2\eta\right|}{M_0u^{p-1}(v+1)}.
\end{align*}
Hence,
\begin{align*}
\frac{\log_2\left(m(U)\right)}{\log_2 R}\ge s-\frac{u^p-1}{u^{p-2}(u-1)}\cdot\frac{v-1}{v+1}-\frac{s}{u^{p-1}}-\frac{\log_2\left(\frac{21^sa_2}{2^sa_1}\right)+p\left|\log_2\eta\right|}{M_0u^{p-1}(v+1)}.
\end{align*}
Let $t<s-u(v-1)/(v+1)$ be fixed. Observe that $\lim_{p\rightarrow\infty}(u^p-1)/u^{p-2}(u-1)=u$, so there exists an integer $P_1>0$ such that 
$$\left|\frac{u^p-1}{u^{p-2}(u-1)}-u\right|\le \frac{1}{2}\left(s-u\cdot\frac{v-1}{v+1}-t\right)\cdot\frac{v+1}{v-1} ,$$
holds for any $p\ge P_1$. Additionally, since
$$\lim_{p\rightarrow\infty}\left(\frac{s}{u^{p-1}}+\frac{\log_2\left(\frac{21^sa_2}{2^sa_1}\right)+p\left|\log_2\eta\right|}{M_0u^{p-1}(v+1)}\right)=0,$$
there exists an integer $P_2>0$ such that the following inequality holds for any $p\ge P_2$:
$$\frac{s}{u^{p-1}}+\frac{\log_2\left(\frac{21^sa_2}{2^sa_1}\right)+p\left|\log_2\eta\right|}{M_0u^{p-1}(v+1)}<\frac{1}{2}\left(s-u\cdot\frac{v-1}{v+1}-t\right).$$
Then, for $p\ge \max\{P_1,P_2\}$, we have
\begin{align*}
\frac{\log_2\left(m(U)\right)}{\log_2 R}&\ge s-\left(u+\frac{1}{2}\left(s-u\cdot\frac{v-1}{v+1}-t\right)\cdot\frac{v+1}{v-1}\right)\cdot\frac{v-1}{v+1}-\frac{1}{2}\left(s-u\cdot\frac{v-1}{v+1}-t\right)\\
&\ge t.
\end{align*}
Let $P=\max\{P_1,P_2\}+1$, and we take $\epsilon=2^{-M_P(v+1)}$. The above computation shows that if $R=\mathrm{diam}(U)<\epsilon$, then
\begin{align*}
m(U)\le R^{t}.
\end{align*}
This finishes the proof of the proposition.
\end{proof}
Now we prove Theorem \ref{main thm}.
\begin{proof}[Proof of Theorem \ref{main thm}]Let $\kappa$ be the constant as in Theorem \ref{thm:BHZ}. Notice that
 $$\lim_{v\rightarrow 1+}\frac{\kappa(v+1)}{4}=\frac{\kappa}{2}>0=\lim_{v\rightarrow 1+}(v-1),$$
there exists $\tilde{v}>1$ such that $\kappa(v+1)/4>v-1$ when $1<v<\tilde v$. Then, for $1<v<\tilde v$, by Lemma \ref{MDP} and Proposition \ref{bound for mU}, there exists a positive constant $u$ such that
 \begin{align*}
 \dim_H(\lim E_n)\ge t
 \end{align*}
 for all $t<s-u(v-1)/(v+1)$. Taking the limit $t\uparrow s-u(v-1)/(v+1)$, we then obtain
 $$\dim_H(\lim E_n)\ge s-u\cdot\frac{v-1}{v+1}.$$
 Take $C=u$. By Proposition \ref{main construction}, $\lim E_n$ is a subset of $W(v)\cap K$. It then follows that
 $$\dim_H(W(v)\cap K)\ge\dim_H(\lim E_n)\ge s-C\cdot\frac{v-1}{v+1}.$$
\end{proof}

\bibliographystyle{plain}
\bibliography{mybibliography}

\begin{thebibliography}{10}

\bibitem{BV}
Victor Beresnevich and Sanju Velani.
\newblock A mass transference principle and the {D}uffin-{S}chaeffer conjecture for {H}ausdorff measures.
\newblock {\em Ann. of Math. (2)}, 164(3):971--992, 2006.

\bibitem{Besicovitch}
A.~S. Besicovitch.
\newblock Sets of fractional dimensions (iv): On rational approximation to real numbers.
\newblock {\em Journal of The London Mathematical Society-second Series}, pages 126--131, 1934.

\bibitem{Bugeaud2016}
Yann Bugeaud and Arnaud Durand.
\newblock Metric diophantine approximation on the middle-third cantor set.
\newblock {\em J. Eur. Math. Soc.}, 18(6):1233--1272, 2016.

\bibitem{BHZ2025}
Timothée Bénard, Weikun He, and Han Zhang.
\newblock Khintchine dichotomy and schmidt estimates for self-similar measures on $\mathbb{R}^d$.
\newblock {\em Preprint arXiv:2508.09076}, 2025.

\bibitem{BHZ}
Timothée Bénard, Weikun He, and Han Zhang.
\newblock Khintchine dichotomy for self-similar measures.
\newblock {\em Preprint arXiv:2409.08061}, 2025.

\bibitem{CVY}
Sam Chow, Péter Varjú, and Han Yu.
\newblock Counting rationals and diophantine approximation in missing-digit cantor sets.
\newblock {\em Preprint arXiv:2402.18395}, 2024.

\bibitem{Datta2024}
Shreyasi Datta and Subhajit Jana.
\newblock On fourier asymptotics and effective equidistribution.
\newblock {\em Preprint arXiv:2407.11961}, 2024.

\bibitem{Fal}
Kenneth Falconer.
\newblock {\em Fractal Geometry}.
\newblock Wiley, 2003.

\bibitem{Hut}
John~E. Hutchinson.
\newblock Fractals and self similarity.
\newblock {\em Indiana University Mathematics Journal}, 30(5):713--747, 1981.

\bibitem{Jarnik}
Vojt{\v{e}}ch Jarn{\'{i}}k.
\newblock Diophantische approximationen und hausdorffsches mass (german) [the diophantine approximation and the hausdorff measure].
\newblock {\em Matem. sb.}, 36:371--382, 1929.

\bibitem{KL}
Osama Khalil and Manuel Luethi.
\newblock Random walks, spectral gaps, and {K}hintchine's theorem on fractals.
\newblock {\em Invent. Math.}, 232(2):713--831, 2023.

\bibitem{Khintchine}
Alexandre Khintchine.
\newblock Einige s{\"{a}}tze {\"{u}}ber kettenbr{\'{u}}che, mit anwendungen auf die theorie der diophantischen approximationen.
\newblock {\em Math. Ann.}, 92:115--125, 1924.

\bibitem{KLW}
Dmitry Kleinbock, Elon Lindenstrauss, and Barak Weiss.
\newblock On fractal measures and diophantine approximation.
\newblock {\em Selecta Mathematica}, 10(479), 2005.

\bibitem{LSV}
Jason Levesley, Cem Salp, and Sanju~L. Velani.
\newblock On a problem of k. mahler: Diophantine approximation and cantor sets.
\newblock {\em Math. Ann.}, 338:97--118, 2007.

\bibitem{Mahler}
Kurt Mahler.
\newblock Some suggestions for further research.
\newblock {\em Bull. Aust. Math. Soc.}, 29:101--108, 1984.

\bibitem{Shmerkin}
Pablo Shmerkin.
\newblock Projections of self-similar and related fractals: a survey of recent developments.
\newblock {\em Preprint arXiv:1501.00875}, 2015.

\bibitem{Weiss2001}
Barak Weiss.
\newblock Akmost no points on a cantor set are very well approximable.
\newblock {\em Proc. R. Soc. Lond. A.}, 457:949--952, 2001.

\bibitem{Yu}
Han Yu.
\newblock Rational points near self-similar sets.
\newblock {\em Preprint arXiv:2101.05910}, 2021.

\end{thebibliography}

\end{document}